\documentclass{amsart}

\usepackage{latexsym,amsmath,amssymb,amsfonts,amscd,graphics}
\usepackage[mathscr]{eucal}
\usepackage{mathrsfs}
\usepackage{tikz-cd}
\usepackage{hyperref}
\hypersetup{
	colorlinks=true,
	linkcolor=blue,
	filecolor=blue,      
	urlcolor=cyan,
	citecolor=magenta}

\numberwithin{equation}{section}
\theoremstyle{plain}
\newtheorem{theorem}[equation]{Theorem}

\newtheorem{lemma}[equation]{Lemma}
\newtheorem{proposition}[equation]{Proposition}
\newtheorem{corollary}[equation]{Corollary}

\theoremstyle{remark}
\newtheorem{remark}[equation]{Remark}

\theoremstyle{definition}
\newtheorem{definition}[equation]{Definition}

\newcommand{\bP}{\mathbb{P}}

\newcommand{\bQ}{\mathbb{Q}}
\newcommand{\bZ}{\mathbb{Z}}
\newcommand{\bF}{\mathbb{F}}
\newcommand{\bC}{\mathbb{C}}

\newcommand{\calC}{\mathcal{C}}

\newcommand{\calH}{\mathcal{H}}

\newcommand{\calM}{\mathcal{M}}
\newcommand{\calO}{\mathcal{O}}

\newcommand{\calD}{\mathcal{D}}

\newcommand{\Aut}{\mathrm{Aut}}

\newcommand{\Sat}{\mathrm{Sat}}

\newcommand{\Ker}{\mathrm{Ker}}

\newcommand{\rank}{\mathrm{rank}}

\newcommand{\Lin}{\mathrm{Lin}}

\newcommand{\Br}{\mathrm{Br}}

\newcommand{\git}{/\kern-0.2em/}
\newcommand{\Lneg}{L_{-}}
\newcommand{\Linv}{L_+}
\author{Lisa Marquand}
\title[]{Cubic Fourfolds with an Involution} 
\date{\today}
\thanks{This work was partially supported by NSF Grant DMS-2101640 (PI Laza)}
\begin{document}
	\bibliographystyle{halpha}
	\maketitle
	\begin{abstract}
		There are three types of involutions on a cubic fourfold; two of anti-symplectic type, and one symplectic. Here we show that cubics with involutions exhibit the full range of behaviour in relation to rationality conjectures. Namely, we show a general cubic fourfold with symplectic  involution has no associated $K3$ surface and is conjecturely irrational. In contrast, a cubic fourfold with a particular anti-symplectic involution has an associated $K3$, and is in fact rational. We show such a cubic is contained in the intersection of all non-empty Hassett divisors; we call such a cubic Hassett maximal. We study the algebraic and transcendental lattices for cubics with an involution both lattice theoretically and geometrically. 
	\end{abstract}
	
	\section{Introduction}
	Cubic fourfolds are one of the central objects of interest in algebraic geometry. In particular, they are intensely studied with respect to rationality. It is expected that the general cubic fourfold is irrational; in contrast, many examples of cubic fourfolds are known to be rational (\cite{hassett}, \cite{Kuznet}, \cite{MR3968870}, \cite{MR3934590}). The moduli and geometry of a cubic fourfold is controlled by a period map, similar to the theory of $K3$ surfaces. A standard approach to rationality questions is Hodge theory (\cite{MR1086707}, \cite{MR1710164}, \cite{hassett}). It is conjectured that a cubic fourfold is not rational if it does not have an associated $K3$ surface (i.e the transcendental cohomology is not induced from a $K3$ surface)(\cite{hassett}, \cite{Kuznet}). In a different direction, cubic fourfolds are closely related to the study and construction of hyperk\"ahler manifolds (e.g. see \cite{MR818549}, \cite{LSV}, \cite{Lehn}).
	
	More recently, there has been renewed interest in the classification of automorphisms of a cubic fourfold. The model for this study is the case of $K3$ surfaces; involutions of $K3$ surfaces have been completely classified by Nikulin \cite{nikulink3}. This was extended to higher order symplectic automorphisms of $K3$ surfaces by Mukai and Kond\=o (\cite{ MR958597}, \cite{MR1620514}). We say an automorphism is symplectic if the automorphism preserves the symplectic form. Similar techniques were adapted to the case of hyperk\"ahler manifolds of $K3^{[n]}$ type, and a classification of symplectic prime order automorphisms was achieved \cite{MR3473636}. Fixed loci of these involutions have also been identified \cite{kamenova2018symplectic}. One can further adapt these techniques to the case of automorphisms of a cubic fourfold $X$. The Strong Torelli theorem (\cite{voisintorelli} \cite{ZHENG_2019}) asserts that automorphisms of $X$ are equivalent to (polarized) Hodge isometries of the middle cohomology $H^4(X,\bZ)$. Such an isometry in turn determines the lattice of algebraic primitive cycles $A(X)_{prim}:=H^{2,2}(X,\bC)\cap H^4(X,\bZ)_{prim}$ contained in the cubic $X$. Similarly to the hyperk\"ahler situation, an automorphism of a cubic fourfold is symplectic if the induced automorphism on $H^4(X,\bZ)$ acts trivially on $H^{3,1}(X)$. Possible cyclic groups of symplectic automorphisms were classified \cite{MR3426426}, followed by a complete classification of possible symplectic automorphism groups \cite{laza2019automorphisms}. Further the lattice $A(X)_{prim}$ was identified in each symplectic case. 
	
The purpose of this paper is to study the case of involutions of a cubic fourfold $X$. Any automorphism of $X$ can be lifted to one of the ambient projective space: there are three possible involutions of $X$, denoted $\phi_1,\phi_2,\phi_3$. Here $\phi_i$ is uniquely characterised by the dimension of the fixed linear subspaces of $\bP^5$; $\phi_i$ fixes complementary linear spaces of codimension $i, 6-i$ respectively. The involutions $\phi_1, \phi_3$ are easily seen to be anti-symplectic, where as $\phi_2$ is symplectic. It is easy to write down the equation for a cubic fourfold with specified involution $\phi_i$. Identifying the lattice $A(X)_{prim}$ and the transcendental lattice $T(X):=(A(X)_{prim})^\perp$ is more subtle; it roughly corresponds to identifying a basis of algebraic cycles. The main result is the classification of these lattices.

\begin{theorem}\label{mainthm}
	Let $X$ be a  general cubic fourfold with $\phi_i$ an involution of $X$ fixing a linear subspace of $\bP^5$ of codimension $i$. Then either:
	\begin{enumerate}
		\item $i=1$, $\phi_1$ is anti-symplectic and $A(X)_{prim}\cong E_6(2)$, $T(X)\cong U^2\oplus D_4^3.$ The algebraic lattice is spanned by classes of planes contained in $X$;
		\item $i=2$, $\phi_2$ is symplectic and $A(X)_{prim}\cong E_8(2)$, $T(X)\cong A_2\oplus U^2\oplus E_8(2).$ The algebraic lattice is spanned by classes of cubic scrolls contained in $X$;
		\item $i=3$, $\phi_3$ is anti-symplectic and $$A(X)_{prim}\cong M,\,\, T(X)\cong U\oplus A_1\oplus A_1(-1)\oplus E_8(2).$$ The algebraic lattice contains an index 2 sublattice spanned by classes of planes contained in $X$.
	\end{enumerate}
Here $M$ is the unique rank 10 even lattice obtained as an index 2 overlattice of $D_9(2)\oplus\langle 24\rangle$, as described in \S \ref{identify all lattices}.
\end{theorem}
Case (1) is studied in detail in \cite{LPZ}; the existence of the involution $\phi_1$ is equivalent to the cubic having an Eckardt point. In case (2), the lattice $A(X)_{prim}\cong E_8(2)$ was identified as part of \cite{laza2019automorphisms} by purely lattice theoretic considerations. Here we show that this lattice is generated by classes corresponding to cubic scrolls contained in $X$; in particular $X$ contains 120 pairs of families of cubic scrolls. Case (3) was previously unknown, we approach this case using both lattice theoretic and geometric techniques.

Our interest in cubics with an involution is twofold. On the one hand, they are relevant to the study of automorphisms of hyperk\"ahler manifolds, especially of $OG10$ type by using the \cite{LSV} construction; this will be discussed elsewhere. One could also study induced automorphisms on the Fano variety of lines (see \cite{MR2967237}). On the other hand, they are interesting in terms of rationality questions for cubic fourfolds. Hassett \cite{hassett} defined a countably infinite union of irreducible divisors $\calC_d$ in the moduli space of cubic fourfolds, parametrising \textit{special} cubic fourfolds (\S \ref{moduli}). For cubics $X\in \calC_d$ for some known values of $d$ \cite{hassett}, there exists an \textit{associated} $K3$ surface whose primitive cohomology can be embedded Hodge-isometrically in $H^4(X,\bZ)_{prim}$ (up to a sign and a Tate twist) \cite{hassett}. It was conjectured by Harris that if the transcendental cohomology $T(X)$ is not induced from a K3 surface, then $X$ is not rational; evidence was given by Hassett \cite{hasset2, hassettsurvey}. Indeed, all known rational cubic fourfolds have an associated $K3$ surface. Kuznetsov \cite{kuzcubic} proposed another criteria for rationality via the derived category, which was verified for the known examples of rational cubic fourfolds \cite{Kuznet}. This criteria is satisfied exactly when $X$ has an associated $K3$ (\cite{addington}, \cite{MR4292740}).  A more generalised notion is an associated \textit{twisted} $K3$ surface $(S,\alpha)$, where $\alpha\in Br(S)$ (\S\ref{associated k3}) (\cite{ k3catofcubic},\cite{emma}). Our next result identifies which cubic fourfolds with an involution have an associated $K3$ or twisted $K3$ surface. 
\begin{theorem}\label{assoc k3 cor}
	Let $X$ be a cubic fourfold with an involution $\phi$ fixing a linear subspace of $\bP^5$ of codimension $i$.
	\begin{enumerate}
		\item For the symplectic involution $\phi:=\phi_2$, $X$ does not have an associated $K3$ surface, or a visible twisted $K3$ surface. 
		\item For the antisymplectic involution $\phi:=\phi_1$, $X$ has an associated twisted $K3$ surface $(S,\alpha)$ for $\alpha\in Br(S)_2$, but does not have an associated $K3$ surface.
		\item For the antisymplectic involution $\phi:=\phi_3$, $X$ has an associated $K3$ surface. 
	\end{enumerate}
\end{theorem}


\begin{remark}
	Cubic fourfolds with large algebraic lattice $A(X):=H^{2,2}(X)\cap H^4(X,\bZ)$ have rich geometry, and thus are natural candidates for rationality questions. Particularly interesting are those that do not have an associated $K3$, as they are potential counterexamples to the  Hassett and Kuznetsov rationality conjectures. Cubic fourfolds with automorphisms naturally have large algebraic lattices, but unfortunately very few symmetric cubics do not have an associated $K3$ surface. This follows from \cite[Proposition 2.9]{laza2018maximally}; if $rank(A(X))> 10$, the transcendental lattice $T(X)$ embeds into the $K3$ lattice.  Further, if the group of symplectic automorphisms of $X$ is not trivial or isomorphic to $\bZ/2\bZ$, then there exists an associated $K3$ surface \cite{ouchi}. Thus the case of the involutions $\phi_1, \phi_2$ are rather special in this sense; conjecturely irrational, but with large algebraic lattices.
\end{remark}
Case (1) and (2) of Theorem \ref{mainthm} are conjecturely irrational, whereas a cubic $X$ with involution $\phi_3$ is potentially rational. We indeed verify rationality by examining which Hassett divisors such a cubic fourfold belongs to.
\begin{definition}
	We say that a cubic fourfold $X$ is \textit{\textbf{Hassett maximal}} if $$X\in \bigcap_{\substack{d>6\\ d\equiv 0,2\,(\mathrm{mod }\, 6)}} \calC_d.$$ We denote the locus of Hassett maximal cubic fourfolds by $\mathcal{Z}$.
\end{definition}
\begin{lemma}\label{hassmax}
	A Hassett maximal cubic fourfold is rational.
\end{lemma}
\begin{proof}
	A Hassett maximal cubic fourfold necessarily belongs to the divisor $\mathcal{C}_{14}$; we can conclude that such a cubic is rational by results of \cite{MR818549} and the fact that rationality specialises in families \cite{MR3987175}.
	\end{proof}
For a cubic $X$ with the involution $\phi_3$, we describe the lattice $A(X)_{prim}\cong M$ explicitly. We can consider the moduli space of cubic fourfolds with the involution $\phi_3$ as $M$- polarized cubic fourfolds \cite{yang2021lattice}. Denote this moduli space $\calM_{\phi_3}$; it is irreducible and 10-dimensional. It turns out that $\calM_{\phi_3}$ is contained in the locus of Hassett maximal cubic fourfolds.
	\begin{corollary}
		Let $\calM_{\phi_3}$ be the moduli space of cubic fourfolds with the involution of type $\phi_3$. Then $\calM_{\phi_3}$ is contained in the locus of Hassett maximal cubic fourfolds $\mathcal{Z}$. In particular, $[X]\in \mathcal{M}_{\phi_3}$ is rational.
	\end{corollary} 
The intersection of all Hassett divisors $\calC_d$ first appears in \cite{yang2021lattice}. They exhibit a component of dimension at least 13 by lattice theoretic considerations. However, the only geometric example of a Hassett maximal cubic, identified in \cite{yang2021lattice}, is the Fermat cubic. Our example of cubic fourfolds with anti-symplectic involution of type $\phi_3$ gives a 10 dimensional geometrically meaningful locus of Hassett maximal cubics.

	\subsection*{Structure of the paper} In \S\ref{prelims}, we recall the necessary results on automorphisms of cubic fourfolds and set up notation. In \S\ref{geometry} we investigate the geometry of a cubic with an involution. We  begin the proof of Theorem \ref{mainthm} by identifying some algebraic cycles contained in a cubic fourfold $X$ with an involution. Exploiting the geometry of the involutions, we show that the lattice $A(X)$ is generated (over $\bQ$) by the classes of planes in the case of an anti-symplectic involution. In \S\ref{sympl inv}, we show that in contrast, a general cubic with symplectic involution cannot contain a plane, and instead we describe $A(X)_{prim}$ in terms of cubic scrolls.
	
	 In \S\ref{proof of main theorem} we restrict our study to the remaining anti-symplectic involution $\phi_3$. We identify $A(X)_{prim}, T(X)$ by first identifying a primitive sublattice generated by differences of classes of planes inside $X$. Using Nikulin's theory of overlattices (see Appendix \ref{appendix}) we show the existence of a class $x\in A(X)$ whose existence is not geometrically obvious. This allows us to identify both $A(X)$ and $A(X)_{prim}$. Finally in \S\ref{rationality} we discuss the implications of rationality for this family of cubic fourfolds. Some relevant results concerning lattices are reviewed in Appendix \ref{appendix}. Appendix \ref{twisted} discusses the intersection of two cubic scrolls from a geometric perspective.
	
	\subsection*{Acknowledgment} I would like to thank my advisor, Radu Laza, for suggesting to me this problem as well as for the invaluable advice and suggestions. Without his support, this paper would not be possible. I would also like to thank Olivier Martin, for valuable discussions regarding the Hassett divisors. I would also like to thank Sebastian Casalaina-Martin, Brendan Hassett, Giulia Saccà, Zheng Zhang and Zhiwei Zheng for some helpful suggestions to improve the original manuscript of this paper.  This work was partially supported by NSF Grant DMS-2101640 (PI Laza).
	
	\section{Preliminaries}\label{prelims}
	In this section we review the relevant results that will allow us to study the involutions of cubic fourfolds. In \S\ref{periods} we review the period map for cubic fourfolds, and in \S\ref{torelli} the Torelli Theorem for cubic fourfolds $X$, and reduce the classification of involutions of $X$ to the classification of involutions of Hodge structures. This is essentially a lattice theoretic question; we follow the approach as in the case of K3 surfaces (Nikulin, Mukai, Kond\=o and others). Next, in \S \ref{inv on lattice} we review lattice theoretic properties of involutions. In \S \ref{geom inv} we recall the geometric classification of involutions. 
	
	\subsection{The Period Map for cubic fourfolds}\label{periods}
	Let $X$ be a smooth cubic fourfold. The middle cohomology $H^4(X,\bZ)$ with the natural intersection pairing is the unique unimodular odd lattice of signature $(21,2)$. Recall that we denote by $\eta_X\in H^4(X,\bZ)$ the square of the hyperplane class of $X$. The primitive cohomology $H^4(X,\bZ)_{prim}:=\langle \eta_X\rangle^\perp$ carries a polarized Hodge structure of K3 type (i.e. Hodge numbers (0,1,20,1,0)). As a lattice, $H^4(X,\bZ)_{prim}\cong L$, where $$L:=(E_8)^2\oplus U^2\oplus A_2.$$ 
	
	Denote by $\calM$ the moduli space of smooth cubic fourfolds, constructed using GIT \cite{laz09}. Denote by $\calD/\Gamma=\{x\in \bP(L\otimes \bC)\mid x^2=0, x\cdot\overline{x}<0\}^+/\Gamma$ the global period domain for cubic fourfolds. Starting with a cubic fourfold $X$, we associate the Hodge structure on its middle cohomology, obtaining a period map:
	$$\mathscr{P}:\calM\rightarrow \calD/\Gamma.$$
	In order to discuss the image of this map, we define two divisors in $\calD/\Gamma$.
	\begin{definition}
		\,
		\begin{enumerate}
			\item A norm $2$ vector $v\in L$ is called a \textbf{short root}. The set of short roots in $L$ determines a $\Gamma$-invariant hyperplane arrangement $\calH_6$ in $\calD$. Let $\calC_6:=\calH_6/\Gamma\subset \calD/\Gamma$ be the associated divisor.	
			\item A norm $6$ vector $v\in L$  with divisibility 3 is called a \textbf{long root}. The set of long roots in $L$ determines a $\Gamma$-invariant hyperplane arrangement $\calH_2$ in $\calD$. Let $\calC_2:=\calH_2/\Gamma\subset \calD/\Gamma$.
		\end{enumerate}
	\end{definition}
	The two divisors $\calC_2, \calC_6$ above are the complement of the image of the period map. Geometrically, $\calC_6$ corresponds to singular cubic fourfolds, where as $\calC_2$ corresponds to degenerations of cubics to the secant to Veronese surface in $\bP^5$.
	
	\begin{theorem}[Voisin, Hassett, Laza, Looijenga]\label{global torelli}
		The period map for cubic fourfolds is an isomorphism 
		\[\mathscr{P}:\calM\rightarrow \calD/\Gamma\setminus(\calC_2\cup \calC_6).\] 
	\end{theorem}

\subsection{Strong Torelli and Automorphisms}\label{torelli}Let $\phi\in \Aut(X)$ be an automorphism of a cubic fourfold $X$. By considering the induced action $\phi^*$ on $H^4(X,\bZ)\cong L$, we obtain a map $$\Aut(X)\rightarrow O(L).$$ We note that this map is injective (combine \cite[Prop. 2.12]{MR3673652}, \cite{autoofhyp}). Further, the Strong Global Torelli Theorem \cite{ZHENG_2019} holds for cubic fourfolds. In other words, any isomorphism between polarized Hodge structures of two smooth cubic fourfolds is induced by a unique isomorphism between the cubic fourfolds themselves. 
	
	\begin{proposition}[Strong Global Torelli Theorem]\label{strong torelli}
		Let $X_1, X_2$ be two smooth cubic fourfolds. Assume that there is an isomorphism $\phi:H^4(X_2,\bZ)\xrightarrow{\cong} H^4(X_1,\bZ)$ of polarized Hodge structures. Then there exists a unique isomorphism $f:X_1\xrightarrow{\cong} X_2$ such that $\phi=f^*$. In particular, for any smooth cubic fourfold $X$,
		\[\Aut(X)\cong \Aut_{HS}(X,\eta_X),\] where $\Aut_{HS}(X, \eta_X)$ denotes the group of Hodge isometries fixing the class $\eta_X$.
	\end{proposition}
	Thus it is natural to study automorphisms of a cubic fourfold via the induced Hodge isometry of $H^4(X,\bZ)_{prim}$.
	
		\begin{definition}\label{symplectic}
		We say an automorphism $\phi\in \Aut(X)$ is \textbf{symplectic} if $\phi$ acts trivially on $H^{3,1}(X)$. Otherwise, $\phi$ is non-symplectic (or anti-symplectic in the case of involutions).
	\end{definition}
Let $X$ be a smooth cubic fourfold and consider the full automorphism group $\Aut(X)$. The induced action of $\Aut(X)$ on $H^{3,1}(X)$ gives a character $\chi:\Aut(X)\rightarrow \bC^\times.$ The kernel is the group of symplectic automorphisms, denoted by $G$. The image of $\chi$ is a cyclic group, denoted by $\bar{G}$. We have the following exact sequence of groups:
$$1\rightarrow G\rightarrow \Aut(X)\rightarrow \bar{G}\rightarrow 1.$$

\begin{proposition}\cite{laza2019automorphisms}
	A prime factor of the order of the automorphism group of a smooth cubic fourfold can only be 2,3,5,7 or 11. Furthermore, a non-symplectic automorphism of prime order can only have order 2 or 3.
\end{proposition}

	\subsection{Moduli of cubics with automorphisms}\label{moduli}
	
	Let $A$ be a finite group, and let $\calM_A$ be the irreducible subvariety parametrising smooth cubic fourfolds with a specified action of $A$, constructed via GIT (see \cite[\S 2]{MR4190414} for details). One can construct a period map for $\calM_A$ in the following way. Let $X\in \calM_A$ and consider the induced action of $A$ on $H^4(X,\bC)$, and let $\xi$ be the character corresponding to $H^{3,1}(X)$. Denote by $H^4(X,\bC)_\xi$ the $\xi$-eigenspace; it admits a natural hermitian form $h$ induced by the intersection pairing of signature $(n,2)$ if $\xi=\bar\xi$, $(n,1)$ otherwise.
\begin{theorem}\cite{MR4190414}
		The Hodge structure on $H^4(X,\bC)_\xi$ gives an algebraic isomorphism $$\mathscr{P}_A:\calM_A\xrightarrow{\cong} (\calD\setminus \calH)/\Gamma.$$ Here $\calD$ is a complex hyperbolic ball if the signature is $(n,1)$; a type IV symmetric domain otherwise. The group $\Gamma$ is an arithmetic group acting properly discontinuously on $\calD$, and $\calH$ is a $\Gamma-$invariant hyperplane arrangement in $\calD$.
\end{theorem}

By the result above, if $A\subset \Aut(X)$ is a group of symplectic automorphisms of $X$, then $\calD$ is a type IV symmetric domain. This is not the case for non-symplectic automorphisms; $\calD$ is a complex hyperbolic ball in general. In the case of anti-symplectic involutions, $\xi=\bar\xi$, and so the resulting period domain is still a type IV symmetric domain.
\begin{corollary}\label{periodmap}
	Let $\calM_\phi$ be the moduli space of smooth cubic fourfolds with a anti-symplectic involution $\phi$. Then we have an isomorphism $$\mathscr{P}_\phi:\calM_\phi\xrightarrow{\cong} (\calD\setminus \calH)/\Gamma$$ as above, where $\calD$ is a component of $\{x\in \bP((L\otimes \bC)_{-})\mid x^2=0, x\cdot \bar{x}<0\}$. Here $(L\otimes\bC)_{-}$ denotes the $-1$ eigenspace of the induced action $\phi$ on $L\otimes \bC\cong H^4(X,\bZ)_{prim}\otimes\bC$.
\end{corollary}
	
	Similarly to the theory developed by Dolgachev \cite{MR1420220} for K3 surfaces, one can also define $\calM_\phi$ as a moduli space of lattice polarised cubic fourfolds. More precisely, let $M$ be a positive definite lattice with a fixed primitive embedding into the primitive lattice $L$. Assume further that $M$ does not contain any short or long roots. One can define a moduli space $\calM_M$ of cubics such that $M\subset H^{2,2}(X)\cap H^4(X,\bZ)_{prim}\subset H^4(X,\bZ)_{prim}\cong L$ and the composition $M\subset L$ is equivalent to the fixed embeddings. In other words, $\calM_M$ is the moduli space of cubics with primitive algebraic lattice equivalent to the prescribed lattice $M$. We say such a cubic is $M$-polarized \cite{yang2021lattice}. Up to passing to a normalization, $\calM_M$ is (possibly the complement of some divisors in) a locally symmetric variety $\calD_M/\Gamma_M$, where $\calD_M$ is the type IV domain associated with the transcendental lattice $T=M_L^\perp$. Thus $\dim\calM_M=20-\rank(M)$. Further, if $M\subset M'\subset L$ (all primitive embeddings), then $\calM_{M'}\subset \calM_M;$ the more algebraic cycles contained in $X$, the smaller the moduli. To construct $\calM_\phi$, we apply the above construction for $M=(L\otimes\bC)_{-}^\perp$. 
	
	In particular, one can consider the loci of \textit{special} cubic fourfolds \cite{hassett} as lattice polarised cubic fourfolds.
	\begin{definition}\label{label}
		A cubic fourfold $X$ is \textbf{special} if it admits a surface $S$ not homologous to a complete intersection, i.e $\bZ[\eta_X]\subsetneq A(X)$. A \textbf{labeling} of $X$ consists of a rank 2 saturated sublattice $K\subset A(X)$ containing $\eta_X$. The discriminant of $K$ is the determinant of the intersection form of $K$, denoted by $d$.
	\end{definition}
	\begin{theorem}\cite{hassett}
		Let $\calM$ denote the moduli of smooth cubic fourfolds. Let $\mathcal{C}_d\subset \calM$ denote special cubic fourfolds admitting a labeling of discriminant $d$. Then $\mathcal{C}_d$ is non-empty if and only if $d>6, d\equiv 0,2 \mod 6$. Moreover, $\mathcal{C}_d$ is an irreducible divisor.
	\end{theorem}
	
	\subsection{Case of Involutions: Arithmetic}\label{inv on lattice}
	Let $L$ be any even lattice. An involution of $L$ determines two eigenspaces
	$$L_{\pm}:=\{v\in L\mid \phi(v)=\pm v\},$$ where $L_{-}$ and $L_+$ are primitive, mutually orthogonal lattices. The group $$H:=L/(L_+\oplus L_{-})$$ is a 2-elementary group, i.e. $H=(\bZ/2\bZ)^a$ for some $a\geq 0$, and admits embeddings into the discriminant groups $A_{L_+}$ and $A_{L_{-}}$ such that the diagonal embedding $H\subset A_{L_+}\oplus A_{L_-}$ is isotropic with respect to the finite quadratic form $q_{L_+}\oplus q_{L_-}$ (see \cite[Sect. 1.5]{nikulin} for more details).
	
	From now on, we consider $L\cong H^4(X,\bZ)_{prim};$ this is an even lattice of signature $(20,2)$ and discriminant group $A_L\cong \bZ/3\bZ$. In contrast to the $K3$ case, this lattice is not unimodular - however the strategy of \cite{nikulin} can be adapted to our situation. 
	
		We define the following sublattices of $H^4(X,\bZ)$:
	\begin{align*}
		A(X)&:=H^{2,2}(X)\cap H^4(X,\bZ),\\
		T(X)&:=(A(X))^\perp_{H^4(X,\bZ)},
	\end{align*} the lattice of algebraic cycles, and the transcendental lattice, respectively. Note that the integral Hodge conjecture holds for cubic fourfolds \cite[Thm 1.4]{MR2993050}, so every class $v\in A(X)$ is indeed algebraic. We denote by $A(X)_{prim}$ the lattice $A(X)\cap H^4(X,\bZ)_{prim}\subset H^4(X,\bZ)_{prim}.$
		
	\begin{proposition}\label{ranks}
		Let $X$ be a cubic fourfold and $\phi\in \Aut(X)$. Let $r$ denote the rank of $L_-$. 
		\begin{enumerate}
			\item The covariant sublattice $L_{-}$ is 2-elementary, with discriminant group of length $0\leq a\leq \min\{r,22-r\}$. Thus
			 $$A_{L_-}\cong (\bZ/2\bZ)^a; A_{L_+}\cong \bZ/3\bZ\oplus (\bZ/2\bZ)^a.$$
			\item If $\phi$ is symplectic, $L_{-}$ has signature $(r,0)$.
			\item If $\phi$ is anti-symplectic, $L_{-}$ has signature $(r-2,2)$. 
		\end{enumerate} 
	\end{proposition}
	\begin{proof}
		Since $\phi\in \Aut(X)\cong \Aut_{HS}(X,\eta_X)$, $\phi$ acts as the identity on $\langle \eta_X\rangle$, and thus acts trivially on $A_{\langle \eta_X\rangle}\cong \bZ/3\bZ\cong A_L$. We see that $H^4(X,\bZ)_{-}\cong L_{-}$, and since $H^4(X,\bZ)$ is an unimodular lattice, $L_-$ is 2-elementary by Lemma \ref{2-elem unimod}.
		The description of the discriminant group of $L_{+}$ follows from Proposition \ref{overlattice}. 
		
		Suppose now that $\phi$ is symplectic; by definition, $\phi$ acts trivially on $T(X)$, thus $L_{-}\subset A(X)$. The class $\eta_X\in A(X)$ is clearly invariant under $\phi$, and so $L_{-}\subset A(X)\cap L$. The claim on the signature follows. Similarly, if $\phi$ is anti-symplectic, then $T(X)\subset L_{-}$, but $\eta_X$ is invariant.
	\end{proof}
	
		\begin{remark}
		Let $X$ be a cubic fourfold with an anti-symplectic involution. Then $L_+\subset A(X)_{prim}$, and we can consider such cubics as latticed polarized cubic fourfolds. Thus the moduli space of cubics with the fixed involution can be constructed as $\calM_{A(X)}$. We say a cubic fourfold with the specified involution is general if $L_+\cong A(X)_{prim}$. The case of $\phi_1$ was studied in detail in \cite{LPZ}, and in the symplectic case \cite{MR4190414}.
	\end{remark}

	\subsection{Case of Involutions: Geometry}\label{geom inv}
	Let $X\subset \bP^{n+1}$ be a smooth hypersurface of degree $d$. Denote by $\Lin(X)$ the subgroup of automorphisms of $X$ consisting of elements induced by projective transformations of the ambient projective space leaving $X$ invariant. By \cite[Theorem 1 and 2]{autoofhyp}, assuming $n\geq2, d\geq 3$ we have that $\Aut(X)=\Lin(X),$ except in the case $n=2, d=4$. Moreover, this group is finite (again excluding the case $n=2, d=4$). 
	
Using the fact that any automorphism can be linearised, one can obtain a complete classification of cyclic automorphisms of cubic fourfolds \cite{Gonz}, and identify which are symplectic \cite{MR3426426}. In particular, for involutions we have the following classification.
	
	\begin{proposition}\label{invofcubic}\cite{Gonz}, \cite{MR3426426}
		Let $X=V(F)$ be a smooth cubic fourfold in $\bP^5$,  $\phi\in \Aut(X)$ an involution of $X$. Applying a linear change of co-ordinates, we can diagonalize $\phi$, so that $$\phi:\bP^5\rightarrow \bP^5, [x_0,\dots x_5]\mapsto [(-1)^{a_0}x_0, \dots, (-1)^{a_5}x_5],$$
		with $a_i\in \{0,1\}$. Let $a:=(a_0,\dots a_5),$ and $d=\dim\calM_\phi$. Then either:
		\begin{enumerate}
			\item $a=(0,0,0,0,0,1)$ and $\phi:=\phi_1$ fixes linear subspaces of $\bP^5$ of codimension 1 and 5 respectively. We have that $\phi_1$is anti-symplectic, $d=14$, and	
			\[F= g(x_0, x_1, x_2, x_3, x_4)+x_5^2l_1(x_0, x_1, x_2,x_3,x_4)\]
			\item $a=(0,0,0,0,1,1)$ and $\phi:=\phi_2$ fixes linear subspaces of $\bP^5$ of codimension 2 and 3 respectively. We have that $\phi_2$ is symplectic, $d=12$ and	 
			\begin{align*}
				F= g(x_0, x_1, x_2, x_3)+&x_4^2l_1(x_0, x_1, x_2,x_3)+\\
				&+x_4x_5l_2(x_0, x_1, x_2,x_3)+x_5^2l_3(x_0, x_1, x_2,x_3)
			\end{align*} 
			\item $a=(0,0,0,1,1,1)$ and $\phi:=\phi_3$ fixes two linear subspaces of $\bP^5$ of codimension 3. We have that $\phi_3$ is anti-symplectic, $d=10$ and	
			\[\qquad	F= g(x_0, x_1, x_2)+x_0q_0(x_3, x_4, x_5)+x_1q_1(x_3,x_4,x_5)+x_2q_2(x_3,x_4, x_5)
			\] 
		\end{enumerate}
	Here $l_i, q_i, g$ denote homogeneous polynomials of degree 1, 2 and 3, respectively.
	\end{proposition}
	\begin{proof}
		The classification and dimension of the families of such cubics can be found in either \cite{Gonz} or \cite{MR3426426}. Let $\Omega=\sum_i(-1)^ix_i dx_0\wedge\dots\wedge \widehat{dx_i}\wedge\dots\wedge dx_5$. A basis for $H^{3,1}(X)$ is given by $\{Res \frac{\Omega}{F^2}\}$, where $F$ is the defining equation for $X$. One can check whether this is invariant or not for each involution above. 	\end{proof}
	
	\begin{corollary}\label{L+ discrim}
		Let $X$ be a general cubic fourfold with an involution $\phi$. Then:
		\begin{enumerate}
			\item If $\phi$ is anti-symplectic, the sublattice $A(X)_{prim}\cong L_+\hookrightarrow L$ is a positive definite lattice of rank \[r(L_+)=\begin{cases}
				6, \text{ if } \phi=\phi_1\\
				10, \text{ if } \phi=\phi_3.
			\end{cases}\]
		\item If $\phi$ is symplectic, the sublattice $A(X)_{prim}\cong L_-\hookrightarrow L$ is a positive definite lattice of rank $8$.
		\end{enumerate} 
	\end{corollary}
	\begin{proof}
		The claims follow immediately from Proposition \ref{ranks} and Proposition \ref{invofcubic}. 
	\end{proof}

	\section{Geometry of anti-symplectic involutions}\label{geometry}
	The purpose of this paper is to identify both the lattice of algebraic cycles and transcendental lattice for a generic cubic fourfold with an involution. We approach this problem via a mixture of geometric considerations and lattice theory. In this section, we prove that in the anti-symplectic case the algebraic lattice is spanned (over $\bQ$) by the classes of planes. It is straightforward to see that a cubic fourfold with an anti-symplectic involution contains many invariant planes. 
	
	\begin{lemma}\label{proj from plane}
		Let $X$ be a general cubic fourfold, $\phi$ an anti-symplectic involution. Then $X$ contains an invariant plane $P$. Further, $(X, P)$ determines a plane sextic curve $C_P$ and a theta-characteristic on $C_P$. More precisely:
		\begin{enumerate}
			\item if $\phi=\phi_1$, $C_P=L\cup Q$ where $L$ is a line, $Q$ is a smooth quintic curve, $\kappa$ a non-trivial odd theta characteristic on $Q$.
			\item if $\phi=\phi_3$ and $C_P=C\cup D$ where $C, D$ are smooth cubic curves, $\kappa$ a non-trivial two torsion line bundle on $D$.
		\end{enumerate}
		Conversely, such a triple $(L,Q, \kappa)$ (respectively $(C,D,\kappa)$) determines a cubic fourfold  with an anti-symplectic involution of type $(X, \phi_1)$ (respectively $(X,\phi_3)$).
	\end{lemma}
	\begin{proof}
		Keeping the notation of Theorem \ref{invofcubic}, we see that we can choose co-ordinates for $\bP^5$ such that $\phi$ is either $\phi_1$ or $\phi_3$ and the equation for $X$ is as in the theorem. 
		
		For the involution $\phi_1$, studied in \cite{LPZ}, $X$ contains the cone over a cubic surface $S$, where $S$ is fixed by the involution. Each of the 27 lines on the cubic surface gives a plane contained in $X$ that is invariant under $\phi_1$.
		
		In the case of $\phi_3$, one can see from the equation of $X$ that the plane $V(x_0,x_1,x_2)$ is contained in $X$, and is point wise fixed by $\phi_3$. 
		
	Choose this invariant plane $P$. The linear projection with center $P$ expresses $X$ as a quadric fibration over $\bP^2$; indeed, $\bP^2$ parametrises the space of $\bP^3$-sections of $X$ containing $P$ (see for example \cite{voisintorelli}).  The blow up $\bP^5_{P}$ of the ambient projective space along $P$ gives a commutative diagram:
		\[ \begin{tikzcd}
			X_{P} \arrow{r} \arrow[swap]{dr}{\pi} & \bP^5_{P} \arrow{d}{\tau} \\%
			& \bP^2
		\end{tikzcd}\]
		where $X_{P}$ is the strict transform of $X$, and $\tau$ and $\pi$ are the linear projections with center $P$. The generic fiber of $\pi$ is a smooth quadric surface; the degenerate fibers of $\pi$ are parametrised by a plane sextic $C_P$, the discriminant curve, where $C_P=V(\det A_i)$ for a matrix $A_i$ depending on $\phi_i$ for $i=1,3$ respectively. A simple analysis of the cases gives:
		
			\[A_1=\left(\begin{matrix}
			l_3 &l_{34} &0&q_3\\
			l_{34}&l_4&0&q_4\\
			0&0&l_1&0\\
			q_3&q_4& 0&f
		\end{matrix}\right);\,\, A_3=\left(\begin{matrix}
		l_3 &l_{34} &l_{35}&0\\
		l_{34}&l_4&l_{45}&0\\
		l_{35}&l_{45}&l_5&0\\
		0&0&0& g
	\end{matrix}\right), \] where each $l_i, l_{ij}$ are linear, $q_i$ quadratic and $f,g$ cubic polynomials. We see that $C_P$ is the union of two curves of the correct degree respectively. We discuss the case of $\phi_3$ in detail; the other case is similar.
		
		Let $X$ be a cubic fourfold with the involution $\phi_3$ - we can rewrite the equation of $X$ as:
			\begin{align}\label{eqnphi3}
				l_3x_3^2+l_4x_4^2+l_5x_5^2+2l_{34}x_3x_4+2l_{35}x_3x_5+2l_{45}x_4x_5 +g=0
			\end{align}
			where $l_i, l_{ij}, g$ are homogeneous polynomials in $x_0, x_1, x_2$, with $l_i, l_{ij}$ linear and $g$ degree 3.  The discriminant sextic $C_P$ is given by the determinant of the matrix $A_3$, and
			so $C_P$ is the union of two smooth cubic plane curves $C$ and $D$, where \begin{align*}\label{cubics}
				C&=V(g)\\
				D&=V(l_3l_4l_5-l_{45}^2l_3+2l_{34}l_{35}l_{45}-l_{34}^2l_5- l_{35}^2l_4).\end{align*}

		Since $D$ is a determinantal curve, by \cite[Prop 4.2]{beauville} this determines an even theta characteristic $\kappa$ on the corresponding curve. In particular, since $D$ is an elliptic curve $\kappa^{\otimes 2}=\calO_D$.
		Conversely, suppose we have the triple $(C, D,\kappa)$. Then by \cite[Prop 4.2]{beauville} we can write $D$ as the determinant of a $3\times 3$ matrix of linear forms in 3 variables. Using this matrix  and $C=V(g)$, we can write an equation for a cubic fourfold of the form of equation (\ref{eqnphi3}). This is clearly invariant under the involution $\phi_3$. 
	\end{proof}
	
	The general fiber of $\pi:X_{P}\rightarrow \bP^2$ is a smooth quadric surface, and the fiber over a smooth point of the discriminant curve $C_P$ is a quadric cone.
	We claim that the fiber over a node of $C_P$ is the union of two planes. This follows from the following elementary lemma, we omit the proof.
	\begin{lemma}
		Let $P\subset X$ be a cubic fourfold containing a plane, $\pi:X_{P}\rightarrow \bP^2$ be the quadric fibration obtained via projection from $P$. Suppose that $\pi^{-1}(p)$ was a plane with multiplicity 2 (a double plane). Then $X$ is a singular cubic fourfold.
	\end{lemma}
	
	\subsection{The involution $\phi_1$} One sees that we have 5 pairs of planes corresponding to the singular points of the discriminant curve $C_P=L\cup Q$. In fact, a cubic fourfold $X$ with the involution $\phi_1$ contains more planes, and we have the following result:
\begin{theorem}\cite{LPZ}
	For a  cubic fourfold with an involution $(X,\phi_1)$ as in \cite{LPZ}, the following hold:
	\begin{enumerate}
		\item $X$ is geometrically equivalent to a pair $(Y,H)$ where $Y$ is a cubic threefold, $H$ a hyperplane in $\bP^4$.
		\item $X$ contains 27 planes $\Pi_i$ passing through the Eckardt point $p$, corresponding to the 27 lines on the cubic surface $Y\cap H$;
		\item The primitive algebraic cohomology $A(X)_{prim}\cong E_6(2)$ (spanned by classes $[\Pi_i]-[\Pi_j])$;
		\item The transcendental cohomology of $X$ is $T\cong (D_4)^3\oplus U^2$.
	\end{enumerate}
\end{theorem}
Their method uses the description of invariant surface classes as cones over the 27 lines on the cubic surface $S=Y\cap H$ with vertex $p$. Indeed, the plane $P$ corresponds to a line on the cubic surface, and each pair of planes corresponds to the residual lines of a tritangent plane containing this line.
	
\subsection{The involution $\phi_3$}
		Denote by $\{F_i,F_i'\}_{i=1}^9$ the $9$ pairs of planes occurring as the fibers of $\pi:X_{P}\rightarrow \bP^2$ over the singular points of the discriminant curve $C_P=C\cup D$. Let the corresponding classes in $H^4(X,\bZ)$ be denoted by $[P],[F_i],[F_i']\in A(X)$ for $i=1,\dots 9$. Notice that $$\eta_X\sim [P]+[F_i]+[F_i'].$$
	\begin{proposition}\label{invariant planes}
		Let $X$ be a cubic fourfold with involution $\phi_3$. Then:
		\begin{enumerate}
			\item $X$ contains at least 19 invariant distinct planes $P, \{F_i, F_i'\}_{i=1}^9$, where additionally $P$ is point wise fixed by $\phi_3$.
			\item The classes $\{\eta_X, [P], [F_1],\dots [F_9]\}\subset H^{2,2}(X)$ span $A(X)\otimes \bQ$. 
		\end{enumerate}
		\end{proposition}
\begin{proof}
	One can check that $\phi_3$ leaves each plane $F_i,F_i'$ invariant be considering equations for these planes.
	
 We have identified 11 linearly independent algebraic classes, thus $rank A(X)\geq 11$. By Corollary \ref{L+ discrim}, we see this must be equality. Thus $A(X)$ is determined up to finite index by these classes.
	\end{proof}
The identification of the lattices $A(X), A(X)_{prim}$ and $T(X)$ is done in \S\ref{proof of main theorem}. In contrast to the previous case, the lattice spanned by the differences of two planes forms an index 2 sublattice of $A(X)_{prim}$.

\section{Geometry of Symplectic Involutions}\label{sympl inv}
	A cubic fourfold $X$ with the symplectic involution  $\phi_2$ was studied briefly as part of \cite{laza2019automorphisms}, where they identified $A(X)_{prim}\cong E_8(2)$ via lattice theory. The geometry was not explored - here we make a couple of complementary remarks, distinguishing this case from the anti-symplectic situation. In particular, we prove that such a cubic cannot contain a plane, and further the lattice $A(X)_{prim}$ is generated by classes that correspond to cubic scrolls contained in $X$.
	
\subsection{Non-existence of planes}	Let $X$ be a general cubic fourfold with a symplectic involution $\phi:=\phi_2$; we first show that $X$ contains no planes. We recall a Hodge theoretic characterisation of the condition to contain a plane. 
	\begin{lemma}\cite{voisintorelli}\label{planes}
		Let $P_1, P_2$ be planes contained in $X$, and denote by $p_i=[P_i]\in H^4(X,\bZ)$. Then:
		\begin{enumerate}
			\item $\eta_X\cdot p=1, $
			\item $p^2=3,$
			\item $p_1\cdot p_2=\begin{cases}
				0 & \text{ if } P_1\cap P_2=\emptyset,\\
				1 &\text{ if } P_1\cap P_2= \text{ a point },\\
				-1 & \text{ if } P_1\cap P_2= \text{ a line}.
			\end{cases}$
		\end{enumerate}
	\end{lemma}
	
	This has the following important consequence:
	\begin{corollary}
		Let $p\in H^4(X,\bZ)$ such that $p^2=3,$ and $ p\cdot \eta_X=1$. Then $p$ is represented by a unique plane.
	\end{corollary}
	One can also detect the existence of a plane by studying the discriminant group of $A(X)_{prim}.$
	\begin{lemma}
		Let $X$ be a cubic fourfold containing a plane $P$. Then $P$ determines a non trivial class $\bar{\delta}\in A_{A(X)_{prim}}$ of order $3$. Moreover, any two planes $P_1$, $P_2$ determine the same class, i.e $\bar{\delta}_1=\bar{\delta}_2\in A_{A(X)_{prim}}$.	\end{lemma}
	\begin{proof}
		Let $p$ denote the class of the plane $P$ in $H^4(X,\bZ)$. Consider the class $\delta:= 3p-\eta_X$; then $\delta\in A(X)_{prim}$ with $\delta^2=24.$   Let $\alpha\in A(X)_{prim}$; we have that $\alpha\cdot(3p-\eta)=3\alpha\cdot p$. Thus $div_{A(X)_{prim}}(\delta)=3$. 
		
		Note that $\delta$ is primitive: if $\delta=kw$ for some $w\in A(X)_{prim}, k\in \bZ$, then $k$ must divide $div_{A(X)_{prim}}(\delta)=3$. On the other hand, $k^2=9$ does not divide $24$; thus $k=1$. 
		
		Thus $\delta$ is primitive of divisibility 3, and so $\delta^*=\frac{\delta}{3}\in A(X)_{prim}^*$, where $$A(X)_{prim}^*=\{y\in A(X)_{prim}\otimes \bQ|y\cdot x \in \bZ \text{ for all } x\in A(X)_{prim}\}.$$ Let $\bar{\delta}$ be the image of $\delta^*$ in $A_{A(X)_{prim}}$; then $\bar{\delta}$ is a nontrivial element of order three.
		
		Now suppose $X$ contains two planes; denote their cohomology classes by $p_1, p_2$. Then in turn this determines $\delta_1^*, \delta_2^*$ as above. Note that $\delta_1^*-\delta_2^*=p_1-p_2\in A(X)_{prim};$ it follows that the images of $\delta_1^*$ and $\delta_2^*$ in $A_{A(X)_{prim}}$ coincide.
	\end{proof}
	
	\begin{corollary}\label{symplectic no plane}
		Let $X$ be a general smooth cubic fourfold with the symplectic involution $\phi$. Then $X$ does not contain a plane. 
	\end{corollary}
	\begin{proof}
		From \cite{laza2019automorphisms}, the general such cubic has $A(X)_{prim}\cong E_8(2)$; the determinant is $2^8$. In particular, there are no non-trivial elements of $A_{A(X)_{prim}}$ of order 3.
	\end{proof}

\subsection{Existence of cubic scrolls}\label{symplscrolls} Let $v\in A(X)_{prim}\cong E_8(2)$ be an element with $v^2=4$. Then $K=\langle \eta_X, v\rangle$ gives a labeling of $X$ with determinant 12; the lattice $E_8(2)$ has 240 such elements. Thus a general cubic fourfold with a symplectic involution belongs to the Hassett divisor $\calC_{12}$; this is the closure of the locus of cubic fourfolds containing a cubic scroll. Equivalently, a general $[X]\in \calC_{12}$ has a hyperplane section with at least 6 double points in linear position \cite[Proposition 23]{flops}. 

\begin{theorem}\label{scrolls in fourfold}
	Let $X$ be a general cubic fourfold with the symplectic involution $\phi$. 
	\begin{enumerate}
		\item $X$ contains 120 pairs of families of cubic scrolls $\{T_i, T_i'\}_{i=1}^{120}$ whose cohomology classes satisfy $[T_i]+[T_i']=2\eta_X$.
		\item The lattice $A(X)_{prim}\cong E_8(2)$ is generated by classes $\alpha_i:=[T_i]-\eta_X$.
		\end{enumerate}
\end{theorem}

In order to prove Theorem \ref{scrolls in fourfold}, we will need to investigate the geometry of such a cubic fourfold $X$ more closely. The involution $\phi$ fixes a line $l\subset X$, and a cubic surface $S:=\Pi\cap X$, where $\Pi\cong \bP^3$ is the complimentary subspace of $\bP^5$ also point wise fixed by the involution (Proposition \ref{invofcubic}).	Let $\pi:Bl_lX\rightarrow \Pi$ be the linear projection from $l$ to the disjoint linear subspace $\Pi$; this is a conic fibration.
The discriminant locus parametrising singular fibers is given by:
\[\det\left(\begin{matrix}
	l_1(x_0,\dots, x_3)& l_2(x_0,\dots, x_3)& 0\\
	l_2(x_0,\dots, x_3)& l_3(x_0,\dots,x_3)& 0\\
	0&0&g(x_0,\dots x_3)
\end{matrix}\right)=0,\]
and is thus the union of the fixed cubic surface $S=V(g)\subset \Pi$ and a quadric surface $Q=V(l_1l_3-l_2^2)$. 
Denote the intersection $S\cap Q=C$; this is then a genus 4 space curve, parametrising the fibers that are double lines.

\begin{lemma}
	Let $X$ be a cubic fourfold with symplectic involution $\phi$, $l$ the point wise fixed line contained in $X$. Let $H\subset \bP^5$ be a general hyperplane containing $l$, and denote $Y:=H\cap X$ the hyperplane section. Then $Y$ is smooth, and the discriminant locus of the linear projection of $Y$ from $l\subset Y$ is the union of a smooth conic $Z$ and a smooth cubic plane curve $E$. 
\end{lemma}
\begin{proof}
	Notations as above. Since $\Pi, l$ are complimentary linear subspaces, $\Gamma:=H\cap \Pi\cong \bP^2$. Note that $\phi$ induces an involution $\phi|_H:H\rightarrow H$ whose fixed locus is $l\cup \Gamma$. The cubic threefold $Y$ is invariant under the involution, and from the equation of $Y$ one can check that $Y$ is smooth. In particular, $Y$ has equation
	\[x_4^2L_1(x_0,x_1,x_2)+2x_4x_5L_2(x_0,x_1,x_2)+x_5L_3(x_0,x_1,x_2)+G(x_0,x_1,x_2)=0,\] where $L_i$ are linear and $G$ has degree 3. Consider the restriction of $\pi:Bl_lX\rightarrow \Pi$ to the proper transform of $Y$; we have a conic bundle $\pi_Y:Bl_lY\rightarrow \Gamma$, with discriminant curve
	\[\det\left(\begin{matrix}
		L_1 & L_2 &0\\
		L_2 & L_3 &0\\
		0& 0& G
	\end{matrix}\right)=0.\] This is the union of the plane cubic $E:=V(G(x_0,x_1,x_2))=S\cap \Gamma$ and $Z:=V(L_1L_3-L_2^2):=Q\cap \Gamma$; both are smooth for a general hyperplane $H$. 
\end{proof}
We claim that if we choose a hyperplane containing $l\subset H\subset \bP^5$ such that the discriminant locus of $\pi_Y$ consists of curves $E,Z$ that are tangent at a point $p$, then $Y:=X\cap H$ is singular.
\begin{lemma}\label{nodes}
	Let $Y$ be a cubic threefold with an involution fixing a line $l$  and consider the discriminant locus $E\cup Z$ of $\pi_Y:Bl_lY\rightarrow \bP^2$ as above. Suppose that the conic $Z$ is tangent to the cubic curve $E$ at a point $p\in\bP^2$. Then $Y$ has (at least) two nodes interchanged by the involution.
\end{lemma}
\begin{proof}
	Without loss of generality, assume that $Z$ is given by $x_0x_1-x_2^2=0$, and suppose that $Z$ and $E$ are tangent at the point $p=[p_0,p_1,p_2]$. Note that this implies that:
	\begin{equation}\label{partials}
		\frac{\partial G}{\partial x_0}(p)=p_1, \, \frac{\partial G}{\partial x_1}(p)=p_0, \, \frac{\partial G}{\partial x_2}(p)=-2p_2.
	\end{equation}
	The equation for $Y$ is given as
	\[x_4^2x_0+2x_4x_5x_2+x_5x_1+G(x_0,x_1,x_2)=0.\] Taking partial derivatives, we see that $Y$ has two nodes, interchanged by the involution, at the points $[p_0,p_1,p_2, \pm\sqrt{-p_1}, \pm\sqrt{-p_0}]$.
\end{proof} 

\begin{proposition}\label{scrolls}
	Let $X$ be a general smooth cubic fourfold with a symplectic involution. Then $X$ contains 120 pairs of families of cubic scrolls ${T_i, T_i^{'}}$ whose classes satisfy $[T_i]+[T_i^{'}]=2\eta_X$. 
\end{proposition}
\begin{proof}
	Notations as above. Let $C\subset \Pi$ be the intersection of $Q\cap S\subset \Pi\subset \bP^5$, the genus $4$ curve as above. Then there are 120 tritangent planes to $C$, denoted by $\Gamma_i\subset \Pi\cong \bP^3$ with $\Gamma_i\cong \bP^2$. Since the intersection points of $\Gamma_i\cap C$ are the intersection points $\Gamma\cap Q\cap S$, we must have that the conic $Z:=\Gamma\cap Q$ is tangent to the cubic curve $E:=\Gamma\cap S$ in three points. Let $H_i=span\{l,\Gamma_i\}\subset \bP^5$; by Lemma \ref{nodes}, $Y_i:=H_i\cap X$ has three pairs of nodes $\{p_i, q_i\}$ with $\iota(p_i)=q_i$. Using the involution, one can check that these nodes are in general position, and so the existence of the cubic scrolls follows by \cite[Proposition 23]{flops}.
\end{proof}

\begin{proposition}\label{root lattice}
	Let $M$ be the lattice spanned by $\{\alpha_i\}_{i=1}^{120}$, where $\alpha_i:=[T_i]-\eta_X$. Then the lattice $M$ is isomorphic to $A(X)_{prim}\cong E_8(2)$.
\end{proposition}
In order to prove Proposition \ref{root lattice}, we must first look at the possible intersection numbers of two cubic scrolls contained in a cubic fourfold. 

\begin{lemma}\label{intersections of scrolls}
	Let $X\subset \bP^5$ be a smooth cubic fourfold containing two non-homologous cubic scrolls $T_1, T_2$. Then $[T_1]\cdot [T_2]=\tau$ for $\tau\in\{1,3,5\}$.
\end{lemma}
\begin{proof}
	The cubic fourfold $X$ has a sublattice $K_\tau:=\langle \eta_X, T_1, T_2\rangle\subset A(X)$, with Gram matrix:
	\[\begin{matrix}
		& \eta_X & T_1 & T_2\\
		\eta_X & 3& 3 & 3\\
		T_1 & 3 & 7& \tau\\
		T_2& 3 & \tau & 7
	\end{matrix}\] for some $\tau\in\bZ$ depending on $X$. The lattice $A(X)$ is positive definite; it follows that the discriminant of $K_\tau$ should be positive. We see that $d(K_\tau)=3(7-\tau^2+6\tau)$, the only values ensuring this is positive are $\tau\in\{0,1,2,3,4,5,6\}$.
	
	Let $\alpha_1:=\eta_X-T_1, \alpha_2:=\eta_X-T_2$; this is a basis for $\langle \eta_X\rangle^\perp_{M_\tau}$. Note that $\alpha_i^2=4$, and $\alpha_1\cdot \alpha_2=\tau-3$. Let $v=x\alpha_1+y\alpha_2$ with $x,y\in \bZ$. We will exclude $\tau=0,2,4,6$ by exhibiting either a short or long root in $\langle \eta_X\rangle^\perp\subset M_\tau$. 
	We have that 
	\begin{equation}
		v^2=2(2x^2+2y^2+xy(\tau-3)).
	\end{equation}
	Let $\tau=0$, then $v=\alpha_1+\alpha_2$ has $v^2=2$. Similarly for $\tau=6$, $v=\alpha_1-\alpha_2$ is a short root.
	Now let $\tau=2$; we see that $v=\alpha_1+\alpha_2$ satisfies $v^2=6$. Note that $v= 2\eta_X-T_1-T_2$; it is easy to check that $v$ has divisibility 3, and is thus a long root. Similarly, for $\tau=4$ we see that $v=\alpha_1-\alpha_2$ is also a long root.
	
	For the remaining values of $\tau$, one can check that $\langle\eta_X\rangle^\perp\subset M_\tau$ contains no long or short roots by using standard Diophantine equation techniques.
\end{proof}

\begin{remark}
Up to orientation, the possible intersection numbers $\tau$ correspond to the intersection numbers of two twisted cubics lying on a cubic surface (see Appendix \ref{twisted}). Notice that the cubic scrolls were obtained by considering odd theta-characteristics of a genus 4 canonical curve. There is a notion of syzygetic/azygyetic theta-characteristics; $\tau=2$ corresponds to when $T_1, T_2$ are obtained from syzygetic odd theta-characteristics (or syzygetic twisted cubics). 
\end{remark}


\begin{proof}[Proof of Lemma \ref{root lattice}]
	Recall that $M$ is the lattice spanned by the classes $\alpha_i:=[T_i]-\eta_X, \alpha_i'=[T_i']-\eta_X$. Since the classes $[T_i], [T_i']$ satisfy $[T_i]+[T_i']=2\eta_X$, we see that $\alpha_i'=-\alpha_i$. Thus $M$ is generated by the classes $\alpha_i$. Each class is contained in $A(X)_{prim}$, and $M$ is an even lattice.  For each $i$ we have $\alpha_i^2=4$, and so each class $\alpha_i\neq 0$. Suppose that $\alpha_i=\alpha_j$. This implies that $[T_i]=[T_j]$, the two cubic scrolls are homologous. By \cite[Lemma 4.1.1]{hassett} this implies that both $T_i, T_j$ are contained in the same hyperplane section of $X$; this contradicts Corollary \ref{scrolls}. Thus we have $120$ distinct classes $\{\alpha_i\}_{i=1}^{120}$ generating $M$.
	
	We next show that the lattice $R:=M(\frac{1}{2})$ is a well-defined integral lattice.
	For $i\neq j$ we see that  \begin{equation}\label{root lattice intersection}
		\alpha_i\cdot \alpha_j=\begin{cases}
		-2\\
		0\\
		2;
	\end{cases}\end{equation} this follows from Lemma \ref{intersections of scrolls}. Recall that $\alpha_i^2=4$; thus the integral lattice $R:= M(\frac{1}{2})$ is well defined.

 The lattice $R$ generated by the 120 classes $\alpha_i$, where $(\alpha_i\cdot \alpha_j)_R:=\frac{1}{2}(\alpha_i\cdot \alpha_j)_M$, and is positive definite. Since $(\alpha_i)_R^2=2$, we can conclude that $R$ is a root lattice with 240 roots, corresponding to $\pm \alpha_i$. Note that the rank of $M$ is at most 8. Since $M=R(2)$, it remains to show that $R\cong E_8$.

     The lattice $E_8$ has 240 roots; suppose for contradiction that $R\not\cong E_8$.  If $R$ is an irreducible root system, then since t$rank(R)\leq 8$, the number of roots of $R$ is at most the number of roots of $D_8$, which is 112. Thus $R$ must be a reducible root system. Let $n(L)$ denote the number of roots of a lattice $L$. Suppose that $R=R_1\oplus R_2$; we have that $n(R)=n(R_1)+n(R_2)$. One can check by direct verification that $n(R)<240$; for example, the number of roots of $D_7\oplus A_1$ is 85. Again we have a contradiction, the only lattice of rank at most 8 with 240 roots is $R:=E_8$. We conclude that $M=E_8(2)\cong A(X)_{prim}$.
	
\end{proof}
	
	\section{The lattices $A(X)_{prim}, A(X),$ and $T(X)$ for $\phi_3$}\label{proof of main theorem}
	In this section we restrict our attention to the study of the the smooth cubic fourfolds $X$ with anti-symplectic involution $\phi:=\phi_3$ as in Theorem \ref{invofcubic}. We will prove Theorem \ref{mainthm} by identifying the lattices $A(X)$, $L_+\cong A(X)_{prim}$, and $T(X)$ for a general such cubic. We briefly outline the strategy. In \S \ref{planes span} we consider the lattice spanned by planes contained in $X$. We show that the primitive lattice spanned by differences of planes $\tilde{K}$ is an index 2 sublattice of $A(X)_{prim}$. In particular, we will show that a class $y=\frac{[P]+\sum_{i=1}^9 [F_i]}{2}$ belongs to $A(X)$. In \S\ref{algebraic lattice} we show that the lattice spanned by $\langle\eta_X, [F_1],\dots [F_9], y\rangle$ is in fact isomorphic to $A(X)$. This allows us to identify the lattice $A(X)_{prim}$ and $T(X)$ in \S \ref{the lattice L+}.
	\newline
	
	For convenience, we collect the lattice invariants for $A(X), A(X)_{prim}$ and $T(X)$ below - these follow from \S \ref{inv on lattice}, \S\ref{geom inv}.
	\begin{lemma}\label{properties}
		Let $X$ be a cubic fourfold with the involution $\phi:=\phi_3$. Then:
	\begin{enumerate}
		\item $A(X)$ is an odd, positive definite, 2-elementary lattice of rank 11 and discriminant group 
		$$A_{A(X)}\cong (\bZ/2\bZ)^a,$$ where $1\leq a\leq 10.$
		\item $A(X)_{prim}$ is a positive definite even lattice of rank 10, with discriminant group 
		$$A_{A(X)_{prim}}\cong \bZ/3\bZ\oplus(\bZ/2\bZ)^a.$$ 
		\item $T(X)$ is a 2-elementary lattice of signature $(10,2)$, with discriminant group 
		$$A_{T(X)}\cong (\bZ/2\bZ)^a.$$
	\end{enumerate}
	\end{lemma}
	\subsection{The lattice spanned by the planes}\label{planes span}
	By Lemma \ref{proj from plane} and Proposition \ref{invariant planes}, $X$ contains the fixed plane $P$, and the invariant planes $F_1,\dots F_9$.
	
	\begin{lemma}\label{intersection matrix}
		The intersection products of the classes above are given as follows:
		\begin{enumerate}
			\item $[P]\cdot [P]=[F_i]\cdot[F_i]=\eta_X\cdot \eta_X=3$ for $1\leq i\leq 9,$
			\item $\eta_X\cdot [P]=\eta_X\cdot [F_i]=1$,
			\item $[P]\cdot [F_i]=-1,$
			\item $[F_i]\cdot[F_j]=1$ for $1\leq i\neq j\leq 9$.
		\end{enumerate}	
	\end{lemma}
	
	\begin{proof}
	One can compute the intersections using Lemma \ref{planes}.
	\end{proof}
	
	\begin{remark}
		Consider the lattice $N=\langle \eta_X, [P], [F_1],\dots [F_9]\rangle\subset A(X)$. This has determinant $2^{12}$, and rank 11. By Lemma \ref{properties}, $N$ cannot be isomorphic to $A(X)$. Indeed, if $N\cong A(X)$, then the number of generators $l(T(X))=a $ of the discriminant group of $T(X)$ satisfies $l(T(X))>10$; this contradicts Corollary \ref{L+ discrim}.
	\end{remark}
	
	\begin{proposition}\label{lattice K}
		Consider the lattice $K$ spanned by the classes $\alpha_i=[F_i]-[F_{i+1}], 1\leq i\leq 8$ and $\alpha_9=[P]+[F_8]+[F_9]-\eta_X$. Then $K$ is a sublattice of $A(X)_{prim}$ and is isomorphic to $D_9(2)$. 
	\end{proposition}
	
	\begin{proof}
		Clearly $\alpha_i\in  A(X)_{prim}$ for $1\leq i\leq 9$; we compute the intersection matrix for the lattice spanned by the $\{\alpha_i\}$ using Lemma \ref{intersection matrix}. In particular, we see that for $1\leq i\leq j\leq 9:$
		\begin{equation*}
				\alpha_i\cdot \alpha_j=\begin{cases}
				4& \text{ if } i=j,\\
				-2 & \text{ if } j=i+1, i\neq 8\\
				-2 & \text{ if } i=7, j=9\\
				0 & \text{ otherwise.}
			\end{cases}
		\end{equation*}
		In other words, we see the lattice spanned by $\{\alpha_i\}_{i=1}^9$ is isomorphic to $D_9(2)$. 
	\end{proof}
	Consider the class $\delta:=\eta_X-3[P]$; we see that $\delta\in A(X)_{prim}$, and the lattices $\langle \delta \rangle$ and $K\cong D_9(2)$ are mutually orthogonal in $A(X)_{prim}$. Thus $A(X)_{prim}$ is an overlattice of $\langle \delta\rangle \oplus D_9(2):$ $$\langle \delta\rangle\oplus D_9(2)\subset A(X)_{prim}.$$   We can calculate that $\delta^2=24$; it follows that $\tilde{K}:=\langle \delta\rangle\oplus D_9(2)\cong \langle 24\rangle\oplus D_9(2).$ The discriminant group of $\tilde{K}$ is
	\[A_{\tilde{K}}=A_{\delta}\oplus A_{D_9(2)}\cong \bZ/24\bZ\oplus \bZ/8\bZ\oplus (\bZ/2\bZ)^8 ,\] whereas $A_{A(X)_{prim}}=\bZ/3\bZ\oplus(\bZ/2\bZ)^a$. By Proposition \ref{overlattice}, $A(X)_{prim}$ is a nontrivial overlattice of $\tilde{K}$, corresponding to a nontrivial isotropic subgroup $H\subset A_{\tilde{K}}$. Since $D_9(2)\hookrightarrow \tilde{K}$ is a primitive embedding, we see that the projections 
	\begin{align*}
		H&\rightarrow A_{\delta}\cong \bZ/24\bZ\\
		H&\rightarrow A_{D_9(2)}\cong \bZ/8\bZ\oplus (\bZ/2\bZ)^8
	\end{align*}are also embeddings.
	It follows that $H\cong \bZ/4\bZ$, or $\bZ/8\bZ$. 
	\begin{lemma}
		There exists a non trivial isotropic subgroup $H\cong \bZ/4\bZ$ of $A_{\tilde{K}}$ corresponding to an overlattice $\tilde{K}\subset M\subset A(X)_{prim}$.
	\end{lemma}
	\begin{proof}
		The discriminant group of $\langle \delta\rangle$ is generated by $\xi=[\frac{\eta_X-3P}{24}], $ and $q_{\delta}(\xi)=\frac{1}{24}.$	Let $G_D$ be the Gram matrix for $D_9(2)$ with respect to the basis $\{\alpha_i\}_{i=1}^9$. 
		To find explicit generators of the discriminant group $A_{D_9(2)}$, we proceed as follows: first consider the inverse matrix $G_D^{-1}$, given below. We consider the linear combinations of $\alpha_1,\dots \alpha_9$ with coefficients given by the columns of $ G_D^{-1}$. Denote them by $\alpha_1^*,\dots \alpha_9^*$, and their image in $A_{D_9(2)}$ by $[\alpha_i^*]$. 
		We see that $\beta:=[\alpha_9^*]$ has order 8, thus we can consider $\beta$ as a generator of $\bZ/8\bZ$. Note that $q_{D_9(2)}(\beta)=\frac{9}{8}$.
		\[G_D^{-1}=\left(
		\renewcommand\arraystretch{1.5}
		\begin{array}{ccccccccc}
			\frac{1}{2} & \frac{1}{2} & \frac{1}{2} & \frac{1}{2} & \frac{1}{2} & \frac{1}{2} & \frac{1}{2} & \frac{1}{4} & \frac{1}{4} \\
			\frac{1}{2} & 1 & 1 & 1 & 1 & 1 & 1 & \frac{1}{2} & \frac{1}{2} \\
			\frac{1}{2} & 1 & \frac{3}{2} & \frac{3}{2} & \frac{3}{2} & \frac{3}{2} & \frac{3}{2} & \frac{3}{4} & \frac{3}{4} \\
			\frac{1}{2} & 1 & \frac{3}{2} & 2 & 2 & 2 & 2 & 1 & 1 \\
			\frac{1}{2} & 1 & \frac{3}{2} & 2 & \frac{5}{2} & \frac{5}{2} & \frac{5}{2} & \frac{5}{4} & \frac{5}{4} \\
			\frac{1}{2} & 1 & \frac{3}{2} & 2 & \frac{5}{2} & 3 & 3 & \frac{3}{2} & \frac{3}{2} \\
			\frac{1}{2} & 1 & \frac{3}{2} & 2 & \frac{5}{2} & 3 & \frac{7}{2} & \frac{7}{4} & \frac{7}{4} \\
			\frac{1}{4} & \frac{1}{2} & \frac{3}{4} & 1 & \frac{5}{4} & \frac{3}{2} & \frac{7}{4} & \frac{9}{8} & \frac{7}{8} \\
			\frac{1}{4} & \frac{1}{2} & \frac{3}{4} & 1 & \frac{5}{4} & \frac{3}{2} & \frac{7}{4} & \frac{7}{8} & \frac{9}{8} \\
		\end{array}
		\right)\]
		
		Let  $H=\bZ/4\bZ= \langle 6\xi+2\beta\rangle$; indeed one can see that $q(6\xi+2\beta)=0\mod 2\bZ$, hence $H$ is an isotropic subgroup, and thus corresponds to some overlattice $\tilde{K}\subset M\subset A(X)_{prim}$.
	\end{proof}
	\begin{proposition}\label{Gram matrix M}
		Notations as above.
		\begin{enumerate}
			\item 	The class \[x=\frac{\alpha_1 +\alpha_3+\alpha_5+\alpha_7+[F_9]-[P]}{2}\] belongs to $M\subset A(X)_{prim}$; in particular $x$ is an integral algebraic class in $A(X)$.
			\item The Gram matrix of $M$ (denoted $G_M$) with respect to the basis $\{x,\alpha_i\}_{i=1}^9$ is given below. In particular, the discriminant group of $M$ is $\bZ/3\bZ\oplus(\bZ/2\bZ)^{10}.$
			\[G_M = \left(\begin{matrix}
				6 & 2 & -2 & 2 & -2 & 2 & -2 & 2 & -2 & 0 \\
				2 & 4 & -2 & 0 & 0  & 0 & 0  & 0 & 0 & 0\\
				-2 &-2 & 4 & -2& 0 & 0 & 0 & 0 & 0 & 0\\
				2  & 0 & -2& 4& -2 & 0 & 0 & 0 & 0 & 0\\
				-2 & 0 & 0 &-2& 4 & -2& 0 & 0 & 0 & 0 \\
				2 & 0 & 0 & 0 & -2& 4& -2 & 0 & 0 & 0\\
				-2 & 0 & 0 & 0 & 0 & -2&4& -2&0& 0\\
				2  & 0 & 0 & 0 & 0 & 0 & -2&4& -2&-2\\
				-2 & 0 & 0& 0 & 0& 0& 0& -2& 4& 0\\
				0 & 0 & 0 & 0 & 0& 0& 0& -2 & 0 & 4
			\end{matrix}\right).\]
		\end{enumerate}
		
	\end{proposition}
	\begin{proof}
		For the first claim, we have that $M$ is the overlattice of $\tilde{K}$ corresponding to the isotrivial subgroup given by $H=\langle 6\xi+2\beta\rangle\subset A_{\tilde{K}}$. We see that 
		\begin{align*}
			6\xi+2\beta&=\frac{ \eta_X -3[P] +2(\alpha_1+\alpha_3+\alpha_5+\alpha_7)-\alpha_8+\alpha_9}{4}\\
			&=\frac{2(\alpha_1+\alpha_3+\alpha_5+\alpha_7)+2[F_9]-2[P]}{4}\\
			&=\frac{\alpha_1 +\alpha_3+\alpha_5+\alpha_7+[F_9]-[P]}{2} \mod \tilde{K}
		\end{align*}
		It follows that class $x$ belongs to $M$; in particular, $x\in A(X)_{prim}$.
		
		For the second claim, 	we can calculate the intersection matrix with respect to the basis $\{x,\alpha_i\}$ using Lemma \ref{intersection matrix}. Let $G_M$ be the matrix with respect to this basis. We see that $\det G_M= 3\times2^{10}$; in order to compute the discriminant group, we compute the inverse matrix $G_M^{-1}$, and consider the column vectors $x^*, \alpha_i^*$. It is clear to see that there are no elements of order 4; thus the discriminant group of $A_M\cong \bZ/3\bZ\oplus (\bZ/2\bZ)^{10}$.
	\end{proof}
	
	\begin{remark}\label{y and eta}
		Consider the linear combination $y:=x+[F_2]+[F_4]+[F_6]+[F_8]+[P]\in A(X)$. We see that 
		\[y:=\frac{[P]+[F_1]+[F_2]+[F_3]+[F_4]+[F_5]+[F_6]+[F_7]+[F_8]+[F_9]}{2}\] is thus an element of the integral lattice $A(X)$.
	\end{remark}

We are now ready to identify the lattices $A(X), A(X)_{prim}, T(X)$; more specifically we prove the following:
\begin{theorem}\label{identify all lattices}
	Let $N$ be the lattice generated by $\{\eta_X,y, [F_1],\dots [F_9]\}$; and $M$ be the lattice generated by $\{x, \alpha_i\} $ for $1\leq i\leq 9$. Then:
	\begin{enumerate}
		\item The lattice of algebraic cycles $A(X)$ is isomorphic to $N$.
		\item The lattice $A(X)_{prim}\cong L_+$ is isomorphic to the lattice $M$.
		\item The lattice $T(X)$ is isomorphic to the lattice $E_8(2)\oplus A_1\oplus A_1(-1)\oplus U$.
	\end{enumerate}
\end{theorem}
This will conclude the proof of Theorem \ref{mainthm}. 

\subsection{The lattice $A(X)$}\label{algebraic lattice} Consider the lattice $A(X)\subset H^4(X,\bZ)$; this is an odd, 2-elementary lattice of rank 11. Let $N$ be the lattice spanned by $\eta_X, y, [F_1],\dots [F_9],$ where $y$ is the class as in Remark \ref{y and eta}; we claim that $N=A(X)$. Let $G_N$ be the Gram matrix of $N$ (with respect to $\eta_X, y, [F_1],\dots [F_9])$, with entries calculated according to Lemma \ref{intersection matrix}.
The inverse matrix is given below. Denote by $\eta^*, y^*, [F_1]^*,\dots [F_9]^*$ the dual basis of $N^*$, given as linear combinations of the elements $\{\eta_X, y, [F_1],\dots [F_9]\}$ with coefficients given by the column vectors of $G^{-1}_N$. By abuse of notation, $\eta^*, y^*,[F_1]^*,\dots [F_9]^*$ also denote the corresponding elements in $A_N=N^*/N$. It is straightforward to check that $A_N$ is isomorphic to $(\bZ/2\bZ)^{10}$ and $\{\eta^*,[F_1]^*,\dots [F_9]^*\}$ is a basis.

\[G_N^{-1}:=\left(
\begin{array}{ccccccccccc}
	\frac{3}{2} & -\frac{5}{2} & 1 & 1 & 1 & 1 & 1 & 1 & 1 & 1 & 1 \\
	-\frac{5}{2} & 6 & -\frac{5}{2} & -\frac{5}{2} & -\frac{5}{2} & -\frac{5}{2} & -\frac{5}{2} & -\frac{5}{2} & -\frac{5}{2} & -\frac{5}{2} & -\frac{5}{2} \\
	1 & -\frac{5}{2} & \frac{3}{2} & 1 & 1 & 1 & 1 & 1 & 1 & 1 & 1 \\
	1 & -\frac{5}{2} & 1 & \frac{3}{2} & 1 & 1 & 1 & 1 & 1 & 1 & 1 \\
	1 & -\frac{5}{2} & 1 & 1 & \frac{3}{2} & 1 & 1 & 1 & 1 & 1 & 1 \\
	1 & -\frac{5}{2} & 1 & 1 & 1 & \frac{3}{2} & 1 & 1 & 1 & 1 & 1 \\
	1 & -\frac{5}{2} & 1 & 1 & 1 & 1 & \frac{3}{2} & 1 & 1 & 1 & 1 \\
	1 & -\frac{5}{2} & 1 & 1 & 1 & 1 & 1 & \frac{3}{2} & 1 & 1 & 1 \\
	1 & -\frac{5}{2} & 1 & 1 & 1 & 1 & 1 & 1 & \frac{3}{2} & 1 & 1 \\
	1 & -\frac{5}{2} & 1 & 1 & 1 & 1 & 1 & 1 & 1 & \frac{3}{2} & 1 \\
	1 & -\frac{5}{2} & 1 & 1 & 1 & 1 & 1 & 1 & 1 & 1 & \frac{3}{2} \\
\end{array}
\right)\]

\begin{lemma}\label{Z10}
	The discriminant group $A_N$ of $N$ is isomorphic to $(\bZ/2\bZ)^{10}$. The discriminant bilinear form \[b_N:A_N\times A_N\rightarrow \bQ/\bZ\] is given by the matrix with every entry equal to $1/2$ (with respect to the basis $\{\eta^*, [F_1]^*,\dots [F_9]^*\}$ of $A_N$).
\end{lemma}
\begin{proposition}\label{lattice A(X)}
The lattice $N$ generated by $\eta_X, y, [F_1],\dots [F_9]$ is saturated in $A(X)=H^4(X,\bZ)\cap H^{2,2}(X)$.
\end{proposition}
\begin{proof}
Assume the natural embedding $N\subset A(X)$ is not saturated. Then it factors as $N\subsetneq \Sat(N)=A(X)\subset H^4(X,\bZ)$, where $\Sat(N)$ denotes the saturation of $N$ in $H^4(X,\bZ)$. Thus, $\Sat(N)$ is a nontrivial overlattice of $N$ and corresponds to a nontrivial isotropic subgroup of $(A_N, b_N:A_N\times A_N\rightarrow \bQ/\bZ)$. Elements in $A_N$ are given by linear combinations of $\eta^*,[F_1]^*,\dots [F_9]^*$ with coefficients either 0 or 1; isotropic elements must have an even number of non-zero coefficients. There are 9 cases to consider. For example, let us suppose that $[F_i]^*+[F_j]^*$ is contained in the isotropic subgroup. From the expression of $G^{-1}_N$ it is easy to see that $[F_i]^*+[F_j]^*\equiv \frac{1}{2}([F_i]+[F_j]) \mod N$. It follows that the element $\frac{1}{2}([F_i]+[F_j])$ belongs to $A(X)$ (i.e $[F_i]+[F_j]$ is divisible by 2). The other cases are similar. In conclusion, $N\neq \Sat(N)$ if and only if at least one of the following elements is 2-divisible in $A(X):$
\begin{enumerate}
	\item $[F_i]+[F_j]$
	\item $[F_i]+[F_j]+[F_k]+[F_l]$
	\item $[F_i]+[F_j]+[F_k]+[F_l]+[F_m]+[F_n]$
	\item $[F_i]+[F_j]+[F_k]+[F_l]+[F_m]+[F_n]+[F_p]+[F_q]$
	\item $\eta+[F_i]$
	\item $\eta+[F_i]+[F_j]+[F_k]$
	\item $\eta+[F_i]+[F_j]+[F_k]+[F_l]+[F_m]$
	\item $\eta+[F_i]+[F_j]+[F_k]+[F_l]+[F_m]+[F_n]+[F_p]$
	\item $\eta+[F_1]+[F_2]+[F_3]+[F_4]+[F_5]+[F_6]+[F_7]+[F_8]+[F_9].$
\end{enumerate}
for $1\leq i, j, k, l, m, n , p, q\leq 9$ distinct. Let us do a case by case analysis.

\begin{enumerate}
	\item Write $[F_i]+[F_j]=2\sigma$ for some $\sigma\in A(X)$. It is easy to see that  $\sigma^2=2$; by \cite[Sect. 4 Prop. 1]{voisintorelli}, this implies that $X$ is a singular cubic fourfold, clearly a contradiction.
	\item Write $[F_i]+[F_j]+[F_k]+[F_l]=2\sigma$, and consider the element $2\sigma -2[F_j]-[2F_l]\in A(X)$. This is clearly divisible by 2; write $2\widetilde{\sigma}=2\sigma-2[F_j]-2[F_l].$ Then $\widetilde{\sigma}^2=2$, again a contradiction.
	\item Write $[F_i]+[F_j]+[F_k]+[F_l]+[F_m]+[F_n]=2\sigma$, and consider the element $2\sigma -2[F_j]-2[F_l]-2[F_n]\in A(X)$. This is clearly divisible by 2; write $2\widetilde{\sigma}=2\sigma -2[F_j]-2[F_l]-2[F_n]$. It is easy to see that $\eta_X\cdot \widetilde{\sigma}=0$, implying that $\widetilde{\sigma}$ is even. On the other hand, $\widetilde{\sigma}^2=3$, a contradiction.
	\item Let $r$ be the index such that $\{i,j,k,l,m,n,p,q, r\}=\{1,\dots 9\}.$ Write $[F_i]+[F_j]+[F_k]+[F_l]+[F_m]+[F_n]+[F_p]+[F_q]=2\sigma$; note that $2y-2\sigma=([P]+[F_r])$. Thus $[P]+[F_r]$ is divisible by 2; write $[P]+[F_r]=2\widetilde{\sigma}$. It is easy to see that $\widetilde{\sigma}\cdot \eta_X=1=\widetilde{\sigma}^2$. Then $X$ is a special cubic fourfold labeled by the rank 2 lattice generated by $\eta_X$ and $\widetilde{\sigma}$. The corresponding discriminant is $d=2$. By \cite[Sect 4.4]{hassett} this cannot happen for the smooth cubic $X$. 
	\item Write $\eta+[F_i]=2\sigma$ for some $\sigma\in A(X)$. Again, $\sigma^2=2$, a contradiction.
	\item Write $\eta+[F_i]+[F_j]+[F_k]=2\sigma$. It is easy to see that $\eta_X\cdot \sigma=3, \sigma^2=6$. Then $X$ is a special cubic fourfold labeled by the rank 2 lattice generated by $\eta_X, \sigma$, with determinant $d=9$. Since $d\neq 0, 2\mod 6$, no such $X$ exists, by \cite{hassett}.
	\item Write $\eta+[F_i]+[F_j]+[F_k]+[F_l]+[F_m]=2\sigma$. Let $n,p,r,s$ be the indices such that $\{i,j, k,l,m,,n,p,r,s\}=\{1,\dots 9\}$. Consider the element \[2\sigma-2y+2[F_s]= \eta-[P]-[F_n]-[F_p]-[F_r]+[F_s];\] we see that $\eta-[P]-[F_n]-[F_p]-[F_r]+[F_s]$ must be divisible by 2. Write $\eta-[P]-[F_n]-[F_p]-[F_r]+[F_s]=2\widetilde{\sigma}$. We see that $\eta_X\cdot\widetilde{\sigma}=0$, and $\widetilde{\sigma}^2=2$, again a contradiction.
	\item Let $r,s$ be the indices such that $\{i,j, k,l,m,,n,p,r,s\}=\{1,\dots 9\}$, and write $\eta+[F_i]+[F_j]+[F_k]+[F_l]+[F_m]+[F_n]+[F_p]=2\sigma$. Consider the element $2\sigma-2y=\eta-[F_r]-[F_s]-[P]\in A(X)$. Thus $\eta-[F_r]-[F_s]-[P]$ is divisible by $2$; write $\eta-[F_r]-[F_s]-[P]=2\widetilde{\sigma}$. One can check that $\eta_X\cdot \widetilde{\sigma}=0$, implying that $\widetilde{\sigma}$ is even. On the other hand, $\widetilde{\sigma}^2=1$, a contradiction.
	\item Write $\eta+[F_1]+[F_2]+[F_3]+[F_4]+[F_5]+[F_6]+[F_7]+[F_8]+[F_9]=2\sigma$. Then $2\sigma-2y=\eta-[P]$; thus $\eta-[P]$ is divisible by 2. Write $2\widetilde{\sigma}=\eta-[P]$. It is clear that $\eta\cdot \widetilde{\sigma}=1=\widetilde{\sigma}^2$, again a contradiction as in (4).
\end{enumerate}
\end{proof}

\begin{corollary}
	We have that $A(X)\cong N$ where $N$ is the lattice given above. In particular, $$A_{A(X)}\cong (\bZ/2\bZ)^{10}.$$
\end{corollary}
\subsection{Proof of Theorem \ref{identify all lattices}}\label{the lattice L+}

The above description of $A(X)$ allows us to identify the discriminant group of $A(X)_{prim}$, and in turn the lattice almost immediately.
\begin{proposition}\label{saturated}
The discriminant group of $A(X)_{prim}$ is isomorphic to $\bZ/3\bZ\oplus (\bZ/2\bZ)^{10}.$ Further, the lattice $A(X)_{prim}$ is isomorphic to the lattice $M$.
\end{proposition}
\begin{proof}
By definition, $A(X)_{prim}=\langle \eta_X\rangle^\perp\subset A(X)\cong N$. Since $A_N\cong (\bZ/2\bZ)^{10}$ and $A_{\langle \eta\rangle}\cong \bZ/3\bZ$, the first claim follows by Proposition \ref{overlattice}. For (2); if $M\neq A(X)_{prim}$, then $A(X)_{prim}$ is a non-trivial overlattice of $M$, corresponding to a non-trivial isotropic subgroup $H\subset A_M\equiv \bZ/3\bZ\oplus (\bZ/2\bZ)^{10}$ by Proposition \ref{overlattice}. This would imply that $A_{A(X)_{prim}}\cong H^{\perp}_{A_M}/H$, which is impossible.
\end{proof}

Next we study the transcendental lattice $T(X)\cong L_{-}$. 
Recall that for a 2-elementary lattice $S$, we define the invariant $\delta(S)\in \{0,1\}$ to be $0$ if $q_S:A_S\rightarrow \bQ/2\bZ$ takes values in $\bZ$, and $1$ otherwise. 
\begin{lemma}\label{transcend}
The invariants of the transcendental lattice $T(X)$ are computed as follows:
\begin{enumerate}
	\item $T(X)$ is an even lattice of rank 12. The signature is $(10,2).$
	\item $A_{T(X)}\cong (\bZ/2\bZ)^{10}$. In particular, $T(X)$ is 2-elementary and $discr(T(X))=1024$.
	\item $T(X)$ has $\delta(T(X))=1$.
\end{enumerate}
The lattice $T(X)$ is isomorphic to the orthogonal direct sum $E_8(2)\oplus A_1\oplus A_1(-1)\oplus U$.
\end{lemma}
\begin{proof}
The first claim follows from Lemma \ref{properties}. Recall that $A_{A(X)_{prim}}=\bZ/3\bZ\oplus (\bZ/2\bZ)^{10}$, and that $A(X)_{prim}\hookrightarrow A_2\oplus U^3\oplus E_8^2$ primitively. By Lemma \ref{embedding general}, this is equivalent to \begin{align*}
q_{A(X)_{prim}}|_{\bZ/3\bZ}&\cong q_{A_2},\\
q_{A(X)_{prim}}|_{(\bZ/2\bZ)^{10}}&\cong - q_{T(X)};
\end{align*} in particular $A_{T(X)}\cong (\bZ/2\bZ)^{10}$. We see that $\delta(T(X))=1$ by computing the values of $q_T=-q_{A(X)_{prim}}|_{(\bZ/2\bZ)^{10}}$.

The indefinite, 2-elementary lattice $T(X)$ is classified uniquely up to isomorphism by the signature, $\delta(T(X)),$ and $l(T(X)),$ where $l(T(X))=10$ is the minimum number of generators of $A_T$ (Theorem \ref{class2-elem}). The lattice $E_8(2)\oplus A_1\oplus A_1(-1)\oplus U$ has the same invariants, and thus they are isomorphic.	
\end{proof} 

\section{Associated $K3$ surfaces, Hassett divisors, and Rationality}\label{rationality}
In this section we investigate the consequences of Theorem \ref{mainthm} in terms of rationality. In \S\ref{associated k3} we investigate the existence of associated $K3$ surfaces; a cubic fourfold is conjectured to be rational if and only if such an associated $K3$ surface exists. Next in \S\ref{C8} we investigate whether a cubic with the anti-symplectic involution $\phi_3$ is trivially rational, i.e contains two disjoint planes. Finally in \S\ref{all hassett} we show that such a cubic fourfold is Hassett maximal; in particular such a cubic is rational.

\subsection{Associated and twisted $K3$ surfaces}\label{associated k3}
Let $X$ be a smooth cubic surface with a labeling $K_d\subset A(X)$, as in Definition \ref{label}.

\begin{definition}
	A polarised $K3$ surface $(S,L)$ of degree $d$ is \textbf{associated} to $X$ if there exists an isomorphism of Hodge structures $$K_d^\perp\cong H^2(S,\bZ)_{prim},$$ where $H^2(S,\bZ)_{prim}$ is orthogonal to $L$ in $H^2(S,\bZ)$. 
\end{definition}

A cubic fourfold $X\in \calC_d$ has such an associated $K3$ surface if and only if $d$ satisfies the following condition \cite{hassett}:
\begin{equation}\label{hassett condition}
	d \text{ even and not divisible by } 4, 9 \text{ or any odd prime } p\equiv 2\mod 3.
\end{equation}
Notice that this implies that as lattices $T(S)\cong T(X)(-1)$, and so a necessary condition for the existence of an associated $K3$ surface is that $T(X)$ embeds primitively into the $K3$ lattice $\Lambda_{K3}\cong U^3\oplus E_8^2$. It is conjectured that a cubic fourfold is rational if and only if there exists an associated $K3$ surface \cite{hassett, kuzcubic}, we investigate the existence of such $K3$'s for cubics with involutions.

\begin{lemma}
	Let $X$ be a general cubic fourfold with involution either $\phi_1$ or $\phi_2$. Then there does not exist an associated $K3$ surface.
	\end{lemma}
\begin{proof}
	Consider the involution $\phi_1$; the statement is proved in \cite[Theorem 1.8]{laza2018maximally}. In fact, the lattice $A(X)_{prim}$ for such a cubic $X$ is maximal in a certain sense.
	
	Consider next the involution $\phi_2$; we show there does not exist a primitive embedding $T(X)\hookrightarrow \Lambda_{K3}$. Recall that in this case $A(X)_{prim}\cong E_8(2)$, and so $T(X)$ has signature $(2,12)$ with discriminant group $A_{T(X)}\cong (\bZ/2\bZ)^8\oplus \bZ/3\bZ$. We also have that $q_{T(X)}|_{\bZ/3\bZ}=q|_{A_2}$, and $q_{T(X)}|_{(\bZ/2\bZ)^8}=-q_{E_8(2)}$.
	
	Suppose that there exists such a primitive embedding. Since $\Lambda_{K3}$ is the unique even unimodular lattice with signature $(3,19)$, by \cite[Prop 1.15.1]{nikulin} the existence of this embedding is equivalent to the existence of an even lattice $K$ of signature $(1,7)$, discriminant group $A_K\cong A_{T(X)}$ such that $q_K=-q_{T(X)}$. Since $(\bZ/2\bZ)^8\subset A_K$, by Lemma \ref{M(1/2) is defined} the lattice $M:=K(1/2)$ is a well-defined integral lattice. Now $M$ is a rank $(1,7)$ lattice with $A_M\cong \bZ/3\bZ$. 
	
	First, suppose that $M$ is an odd lattice. Then by Lemma \ref{T(2) has delta=1}, there exists a generator $\xi\in\bZ/2\bZ\subset A_K$ such that $q_K(\xi)\not\in \bZ/2\bZ.$ Since $K^\perp=E_8(2)$ is a 2-elementary lattice such that $q_{E_8(2)}(v)\in \bZ/2\bZ$ for all $[v]\in A_{E_8(2)}$, this is a contradiction - thus $M$ must be an even lattice. By Theorem \ref{hyperbolic p-elementary}, $M$ is an even $3$-elementary lattice and is uniquely determined by the rank and $l(A_M)=1$. Consider the lattice $U\oplus E_6$. This is a $3$-elementary lattice with the same invariants as $M$; thus $M\cong U\oplus E_6$.	Hence $K\cong U(2)\oplus E_6(2)$; however $q_M|(\bZ/3\bZ)\neq - q_{A_2}$; it follows that no such lattice $K$ exists.	
	\end{proof}
The Kuznetsov conjecture would then imply that a cubic fourfold $X$ with involution $\phi_1$ or $\phi_2$ is irrational. On the other hand, for $X$ with involution $\phi_3$, we can construct a $K3$ surface with transcendental lattice $T(S)\cong T(X)(-1)$ geometrically.
\begin{lemma}
	Let $X$ be a general cubic fourfold with involution $\phi_3$. Then there exists a primitive embedding $T(X)\hookrightarrow \Lambda_{K3}$.
\end{lemma}
\begin{proof}
	First, assume that $\phi=\phi_3$. We show there exists such a primitive embedding by exhibiting a $K3$ surface with transcendental lattice $T(S)\cong T(X)(-1)$. Let $C\subset \bP^2$ be an irreducible general sextic curve with 9 nodes. Let $\pi:\widetilde{\bP}^2\rightarrow \bP^2$ be the blow up of $\bP^2$ at the 9 nodes, and let $r:S\rightarrow \widetilde{\bP}^2$ be the double cover ramified along the strict transform of $C$. Then $S$ is a K3 surface; we claim $T(S)\cong E_8(-2)\oplus U\oplus A_1\oplus A_1(-1)$. 	Let $h=r^*\pi^* (l)$, where $l$ is the class of a line in $\bP^2$, and $e_i$ for $i=1,\dots 9$ the pullback of the exceptional curves of $\pi$. Let $D\in \widetilde{\bP^2}$ be the strict transform of $C$, and denote also by $D$ the ramified curve on $S$. In particular, $2D\sim 6h-\sum_{i=1}^9 2e_i$, and the $NS(S)$ is spanned by the classes $h, e_1,\dots e_9$. Notice that $e_i^2=-2, h^2=2 $ and $h\cdot e_i=0$. Thus the lattice $NS(S)$ is a 2-elementary lattice, with signature $(1,9)$, $l(A_{NS(S)})=10$, and $\delta=1$ (see Theorem \ref{2-elem}). This determines $NS(S)$ uniquely, and $NS(S)\cong E_8(2)\oplus A_1\oplus A_1(-1)$. Now by definition $T(S)=(NS(S))^{\perp}$, and by Lemma \ref{embedding into unimodular} has signature $(2,10), l(A_{T(S)})=10, \delta=1$. This uniquely determines $T(S)$; in particular, $T(S)\cong E_8(-2)\oplus U\oplus A_1\oplus A_1(-1)$. 	
\end{proof}

Recently, Brakkee considered instead associated twisted $K3$ surfaces \cite{emma}. 
Recall that the Brauer group of a scheme $S$ is the group of sheaves of Azumaya algebras modulo Morita equivalence, with multiplication given by tensor product. For references, see \cite{MR2115675, MR2483937}.

For $S$ a complex $K3$ surface, we have that 
$$\Br(S)\cong H^2(S,\calO_S^*)_{tors}\cong (\bQ/\bZ)^{22-\rho(X)}.$$

\begin{definition}
	A \textbf{twisted $K3$ surface} is a pair $(S,\alpha)$ where $S$ is a $K3$ surface, $\alpha\in \Br(S)$.
\end{definition}

More precisely, consider the following condition:
\begin{equation}\label{emma condition}
	d'=dr^2 \text{ for some } d, r \text{ satisfying condition \ref{hassett condition}. }
\end{equation}
\begin{theorem}\cite[Theorem 2]{emma}
	A cubic fourfold $X$ belongs to the divisor $\calC_{d'}$ for $d'$ satisfying (\ref{emma condition}) if and only if for every decomposition $d'=dr^2$, $X$ has an associated polarized twisted $K3$ surface $(S, L,\alpha)$ of degree $d$ and order $r$.
\end{theorem}
Here $r:=order(\alpha)\in Br(S)$. Here $(S,L,\alpha)$ is associated to $X$ if there is a Hodge isometry $K_d^\perp\cong \Ker f_\alpha$. 

Consider the case $d=r=2,$ so $d'=8$. The cubic fourfolds contained in $\calC_8$ contain a plane, and there is an associated twisted $K3$ surface $(S,\alpha)$. As shown by Voisin in \cite{voisintorelli}, there is a geometric construction for $(S,\alpha)$ obtained by projection from the plane, and letting $S$ be the double cover of $\bP^2$ branched in the discriminant sextic. We call this twisted $K3$ the \textbf{visible} twisted $K3$ surface associated to $P\subset X$. By Lemma \ref{proj from plane}, a cubic fourfold with anti-symplectic involution contains a plane - we immediately see the existence of such a $K3$ surface.
\begin{corollary}
	Let $X$ be a cubic fourfold with anti-symplectic involution $\phi$. Then there exists an associated visible twisted $K3$ surface $(S,\alpha)$ with $order(\alpha)=2.$ 
\end{corollary}

On the other hand, let $X$ be a general cubic fourfold with the symplectic involution $\phi_2$. By Corollary \ref{symplectic no plane}, we cannot associate to $X$ a visible twisted $K3$ surface. 

\subsection{The divisor $\calC_8$}\label{C8}
The cubic fourfolds that contain a plane have been well studied and are central to the original proof of the Torelli theorem \cite{voisintorelli}. They have a quadric bundle structure, and rationality would follow provided the bundle has a rational section. Here, we note that this is not the case for cubics $X$ with the involution $\phi_3$.

\begin{definition}\label{trivially rational}
	A cubic fourfold $X$ containing a plane $P$ is called trivially rational (see \cite{MR3602887}) if the associated quadric bundle $\pi:X_P:=Bl_PX\rightarrow \bP^2$ has a rational section.
\end{definition}
In particular, one sees immediately that if a cubic fourfold $X$ contains two disjoint planes, then we have such a section and $X$ is trivially rational. The following theorem gives a criteria for the existence of such a section.

\begin{theorem}\cite[Theorem 3.1]{hasset2}
	A cubic fourfold $X$ containing a plane is trivially rational if and only if there exists a class $T\in A(X)$ with $T\cdot Q$ odd for a smooth fiber $Q$ of $\pi$.
\end{theorem}
We can use this criteria in our situation; our complete description of the lattice $A(X)$ allows us to conclude no such class exists.
\begin{corollary}\label{nosection}
	Let $X$ be a general cubic fourfold with the involution $\phi=\phi_3$, and let $[P]$ be the class of the plane as in \S \ref{proof of main theorem}. Then $X$ is not trivially rational with respect to the quadric fibration $\pi:X_P\rightarrow \bP^2$.
\end{corollary}
\begin{proof}
	Let $Q$ denote a general fiber of $\pi$, i.e a smooth quadric surface. The class $[Q]$ in $H^4(X,\bZ)$ satisfies $[P]+[Q]= \eta_X$. One can easily see that a $\bZ$-linear combination of the basis of $A(X)$ intersects the quadric $[Q]$ evenly, using the intersection matrix in Proposition \ref{lattice A(X)}.
\end{proof}

Let $X\in \calC_8$, and suppose that the discriminant curve $C_P$ associated to $X_P\rightarrow \bP^2$ is smooth. Following Voisin \cite{voisintorelli}, we let $(S,\alpha)$ be associated visible $K3$ surface, where $\alpha\in Br(S)[2]$. This in turn defines a 2-torsion Brauer class $\alpha_X\in \Br(X)$; we have the following result of Kuznetsov \cite[Sect. 4.3]{Kuznet}.
\begin{lemma}
	Let $X$ be a cubic fourfold containing a plane $P$ as above. The following are equivalent:
	\begin{enumerate}
		\item there exists a rational section of the quadric fibration $\pi:X_P\rightarrow \bP^2$;
		\item the associated Brauer class is trivial, i.e. $\alpha_X=1\in \Br(X)$.
	\end{enumerate}
	Moreover, the conditions above imply that $X$ is rational.
\end{lemma}

\begin{corollary}
	Let $X$ be a cubic fourfold with anti-symplectic involution $\phi_3$. Then the associated Brauer class is non-trivial.
\end{corollary}

\subsection{The Intersection of the Hassett divisors}\label{all hassett}
Cubic fourfolds with involutions $\phi_1$ or $\phi_2$ have no associated $K3$ surfaces, and are conjecturely irrational. On the other hand, a cubic fourfold $X$ with the involution $\phi_3$ has transcendental lattice coming from a $K3$ surface, and is potentially rational. Despite the rationality not following from the obvious quadric bundle structure, we will establish that $X$ is indeed rational by investigating which Hassett divisors such an $X$ belongs to. Recall the following definition:
\begin{definition}
	We say that a cubic fourfold $X$ is \textbf{Hassett maximal} if $$X\in \bigcap_{\substack{d>6\\ d\equiv 0,2\,(\mathrm{mod }\, 6)}} \calC_d.$$ We denote the locus of Hassett maximal cubic fourfolds by $\mathcal{Z}$.
\end{definition}
 Cubics contained in the Hassett maximal locus $\mathcal{Z}$ are rational (see Lemma \ref{hassmax}). In this section we prove the following result:
	\begin{theorem}\label{all the hasset}
	Let $\calM_{\phi_3}$ be the moduli space of cubic fourfolds with the involution of type $\phi_3$. Then $\calM_{\phi_3}$ is contained in the Hassett maximal locus
	$\calM_{\phi_3}\subset \mathcal{Z}$ In particular, $X\in \calM_{\phi_3}$ is rational. 
\end{theorem}

It is known that $\mathcal{Z}$ is non-empty; it contains the Fermat cubic fourfold \cite[Theorem 1.2]{yang2021lattice}. Further, the authors show that $\dim{\mathcal{Z}}\geq 13,$ by illustrating a lattice $M$ of rank 7 such that the moduli of $M$-polarized cubic fourfolds $\calM_M$ is non-empty, and $M$ contains a labeling of determinant $d$ for every $d>6, d\equiv 0, 2 \mod 6$. We adapt the method in \cite{yang2021lattice} to our situation. 

In order to prove Theorem \ref{all the hasset}, we will need the following classical results of number theory.
\begin{lemma}(Lagrange's 4-square theorem)
	Any positive integer can be expressed in the form $x^2+y^2+z^2+u^2$ for some integers $x,y,z,u$.
\end{lemma}

\begin{lemma}(Ramanujan)
	Any positive integer except for $1$ and $17$ can be expressed in the form $2x^2+2y^2+2z^2+3u^2$ for some integers $x,y,z,u.$
\end{lemma}
\begin{proof}[Proof of Theorem \ref{all the hasset}]
	We will exhibit a primitive sublattice $\eta_X\in K_d\hookrightarrow A(X)$ with determinant $d$ for every $d>6, d\cong 0,2\mod 6.$
	Recall that a basis for $A(X)$ is given by $$\{\eta_X, y, [F_1],[F_2],\dots [F_9]\}$$ keeping notations of \S\ref{proof of main theorem}. Denote by $\alpha_1=[F_1]-[F_2], \alpha_3=[F_3]-[F_4], \alpha_5=[F_5]-[F_6], \alpha_7=[F_7]-[F_8], \beta=[\eta_X]-[P]-[F_9]$, and $\gamma=y-[F_5]-[F_6]-[F_7]-[F_8]-[F_9].$ It is easy to see that the sublattice lattice $\langle\eta_x, \alpha_1,\alpha_3, \alpha_5,\alpha_7, \beta,\gamma  \rangle\subset A(X)$ is primitive; indeed, writing each class in the basis of $A(X)$ we see they are linearly independent.
	
	Suppose that $v=x_1\alpha_1+x_3\alpha_3+x_5\alpha_5+x_7\alpha_7+s\beta+t\gamma$ for integers $x_1,\dots x_7, s,t$. One can see that $\eta_X\cdot v=s$, and 
	$$v^2=4x_1^2+4x_3^2+4x_5^2+4x_7^2+ 3s^2+6t^2.$$
	Consider the rank two sublattice $\langle \eta_X, v\rangle\subset A(X)$: its discriminant is given by:
	$$d=3(4x_1^2+4x_3^2+4x_5^2+4x_7^2+ 6t^2)+8s^2.$$ We will show we can obtain every $d\equiv 0, 2 \mod 6$. 
	
	\textbf{Case 1: $d=6k$,} for $k\geq 2$. Let $s=0$. We need to find suitable integers such that 
	$$k=2x_1^2+2x_3^2+2x_5^2+2x_7^2+ 3t^2.$$ 
	\begin{itemize}
		\item If $k=2m$, let $x_1=1$. Then the lattice $\langle \eta_X, v\rangle$ is primitive, and by Ramanujan's theorem we can find suitable integers such that $$2(m-1)=2x_3^2+2x_5^2+2x_7^2+ 3t^2.$$
		\item If $k=2m+1$, let $t=1$. Then the lattice $\langle \eta_X, v\rangle$ is primitive, and by Lagrange's 4-square theorem we can find suitable integers such that $$m-1=x_1^2+x_3^2+x_5^2+x_7^2.$$
	\end{itemize}
	
	\textbf{Case 2: $d=6k+2$.} Since we know $[X]\in \mathcal{C}_8$ (X contains a plane), we can assume that $k\geq 2$. Let $s=1$. The lattice $\langle \eta_X, v\rangle$ is primitive, and we need to find suitable integers such that 
	$$k-1=2x_1^2+2x_3^2+2x_5^2+2x_7^2+ 3t^2.$$ For $k\geq 3$ this reduces to Case 1 - we deal with $k=2$ below.
	
	\textbf{Case 3: $d=14$.} Consider the class $v=y-[F_2]-[F_4]-[F_6]-[F_8];$ then $v^2= 5$ and $\eta_X\cdot v=1$. Thus the lattice $\langle \eta_X, v\rangle$ is primitive and of determinant $d=14$. It is well known that any $[X]\in \mathcal{C}_{14}$ is rational \cite{MR818549}.
\end{proof}

\subsection{Low degree classes}
We have seen that a cubic $X$ with the involution $\phi_3$ is  rational by showing it belongs to the Hassett maximal locus. In particular, such a cubic belongs to $\mathcal{C}_{14}$, the closure of the locus of Pfaffian cubic fourfolds. The Pfaffian locus has been well studied; Beauville and Donagi \cite{MR818549} showed that Pfaffian cubic fourfolds are rational. Further, the Pfaffian condition is equivalent to $X$ containing a smooth degree $5$ del Pezzo surface \cite{beauville}.  In this section, we show that such a cubic does indeed belong to the Pfaffian locus inside of $\calC_{14}$. Our argument is lattice theoretic; it would be interesting to realise the rationality of a cubic fourfold $X$ with the involution $\phi_3$ geometrically.

The complement of the Pfaffian locus has been studied by Auel in \cite{auel2021brillnoether}. More precisely, the complement of the Pfaffian locus inside $\calC_{14}$ is contained in the irreducible locus of cubic fourfolds containing two disjoint planes. Using this, description we can prove our next result:

\begin{proposition}
	Let $X$ be a general cubic fourfold with the involution $\phi=\phi_3$. Then $X$ is Pfaffian. 
\end{proposition}

\begin{proof}
	Suppose that such a cubic fourfold $X$ is not Pfaffian. Then by \cite[Theorem 1]{auel2021brillnoether}, $X$ contains two disjoint planes $P,P'$. Consider the class $v=[P]-[P']$; clearly $v\in A(X)_{prim}$. Thus we can write $v=a_0x+\sum_{i=1}^9 a_i \alpha_i$, where $\{x,\alpha_1,\dots,\alpha_9\}$ is a basis for $A(X)_{prim}\cong M$ as in Theorem \ref{identify all lattices}. Let $w\in A(X)_{prim}$; we can also write $w=b_0x+\sum_{i=1}^9 b_i\alpha_i$. Notice that $x\cdot w$ and $\alpha_i\cdot w$ are even, by using the intersection matrix in Proposition \ref{Gram matrix M}. Thus $v\cdot \alpha\in 2\bZ$ for all $\alpha\in A(X)_{prim}$. 
	
	Consider the element $\delta:=3[P]-\eta_X$, we see $\delta\in A(X)_{prim}$. Since $P$ and $P'$ are disjoint, $[P]\cdot[P']=0$ and so $v\cdot \delta= 3[P]^2=9$, a contradiction. Thus no two planes are disjoint. 
\end{proof}

The cubic fourfolds $X$ admitting the involution $\phi_3$ thus belong to the intersection of $\mathcal{C}_8\cap \calC_{14}$. In \cite{MR3968870}, the authors conduct a systematic study of this locus. In particular, they prove that $\calC_8\cap \calC_{14}$ has 5 irreducible components, indexed by the value $[P]\cdot [T]\in\{-1,0,1,2,3\}$, where $P\subset X$ is a plane and $[T]$ is the class of a small OADP surface (for a general point in $\bP^5$ there exists a unique secant line to $T$; see \cite[Def. 1.5]{MR3968870}) such that $[T]^2=10$ and $[T]\cdot \eta_X=4$ \cite[Theorem 3.4]{MR3968870}. 

\begin{corollary}\cite[Corollary 3.5]{MR3968870}
	Let $X\in\calC_{14}$, and $[T]\in H^{2,2}(X,\bZ)$ such that $[T]\cdot\eta_X=4$ and $[T]^2=10$. Then $[T]$ is represented by a small OADP surface $T$ contained in $X$.
\end{corollary}

\begin{theorem}
	Let $X$ be a cubic fourfold with an involution $\phi:=\phi_3$. Then $X$ contains a smooth quartic rational normal scroll.
\end{theorem}
\begin{proof}
	Let $[T]=2\eta_X -y+[F_7]+[F_8]+[F_9]$. One can easily compute that $[T]\cdot \eta_X=4$ and $[T]^2=10$; thus $[T]$ is represented by a small OADP surface. By the proof of \cite[Theorem 3.4]{MR3968870}, there are three possibilities for $T$:
	\begin{enumerate}
		\item $T=S\cup P'$ where $S$ is a cubic rational normal scroll and $P'$ is a plane;
		\item $T$ contains only irreducible components of degree less than or equal to $2$;
		\item $T$ is an irreducible smooth quartic rational normal scroll.
	\end{enumerate}
	
	Let $[P]$ be the class of any plane in $X$; thus $[P]$ is equivalent to a $\bZ$-linear combination of the basis of $A(X)$ given in Proposition \ref{lattice A(X)}. One can compute that $[P]\cdot T=2k$ for some integer $k$. We will use this to rule out $(1)$ and $(2)$ above.
	\begin{enumerate}
		\item Suppose that $T=S\cup P'$. Since $T$ is a small OADP, the surfaces $S$ and $P'$ intersect along a line; thus $[S]\cdot [P']=0$. Hence we see that $[P']\cdot [T]=[P']^2=3$, a contradiction.
		\item Suppose that $T$ contains only irreducible components of degree less than or equal to $2$. By the proof of \cite[Theorem 3.4]{MR3968870}, this implies that $X$ contains a pair of skew planes. One can compute that for one such plane $[P]\cdot [T]=-1$, again a contradiction.
	\end{enumerate}
	Thus $[T]$ is represented by a smooth quartic rational normal scroll.
\end{proof}	
	\appendix
	\section{Lattice Theory}\label{appendix}
	
	\subsection{Notation}
	A lattice of rank $r$ is a free finitely generated $\bZ$-module $N\cong \bZ^r$ equipped with a non-degenerate symmetric bilinear pairing $N\otimes N\rightarrow \bZ$, denoted $e\otimes f\mapsto e\cdot f$. We write $e^2=e\cdot e$.  We assume that $N$ is even ($e^2\in 2\bZ$ for each $e\in N$) unless otherwise stated. 
	
	The lattice $N(n)$ is obtained from $N$ by multiplying the pairing by $n$. The lattice $\langle n\rangle$ denotes the lattice $\bZ$ with $e^2=n$ for a generator $e$. The lattices $A_k, D_k, E_k$ denote the standard positive definite root lattices of rank $k$, and $U$ denotes the hyperbolic plane.
	
	The finite abelian group $A_N:=N^*/N$ is called the discriminant group of $N$, with  induced quadratic form $q_N:A_N\rightarrow \bQ/2\bZ$ (or symmetric bilinear form $b_N:A_N\times A_N\rightarrow \bQ/\bZ$ if $N$ is odd).
	
	\subsection{Overlattices}
	
	Let $L$ be an even lattice. An \textbf{overlattice} of $L$ is an even lattice $\Gamma$ where $L\hookrightarrow \Gamma$ is an embedding for which $\Gamma/L$ is a finite abelian group.
	
	Let $H_\Gamma=\Gamma/L$ We have a chain of embeddings:
	\[L\hookrightarrow \Gamma\hookrightarrow \Gamma^*\hookrightarrow L^*,\]
	and so $H_\Gamma\subset L^*/L=A_L$, and $(\Gamma^*/L)/H_\Gamma=A_\Gamma$.
	
	We have the following result due to Nikulin \cite[Props. 1.4.1, 1.4.2]{nikulin}.
	\begin{proposition}\label{overlattice}
		\begin{enumerate}
			\item The correspondence $\Gamma\mapsto H_\Gamma$ determines a bijection between even overlattices of $L$ and isotropic subgroups of $A_L$ (a subgroup $H\subset A_L$ is isotropic if $q_L|_H=0$).
			\item We have $(H_\Gamma)^\perp = \Gamma^*/L\subset A_L$ and $(q_L|_{(H_\Gamma)^\perp})/H_\Gamma = q_\Gamma.$
			\item Two even overlattices $L\hookrightarrow \Gamma_1$ and $L\hookrightarrow \Gamma_2$ are isomorphic if and only if the isotropic subgroups $H_{\Gamma_1}, H_{\Gamma_2}\subset A_L$ are conjugate under some automorphism of $L$.
		\end{enumerate}
	\end{proposition}
	
	\subsection{Primitive embeddings}
	
	An embedding of lattices $M\subset S$ is called \textit{primitive} if $L/M$ is a free group; we denote this by $M \hookrightarrow L$, and denote $N=M^\perp$ the \textit{orthogonal complement} of $M$ in $L$. 
	
	\begin{lemma}\label{embedding general}\cite[Prop 1.15.1]{nikulin}
		Primitive embeddings of $M$ with signature $(m_+, m_-)$ into an even lattice $S$ with signature  $(s_+, s_-)$ are determined by the sets $(H_M, H_S, \gamma, K, \gamma_K)$ where
		\begin{itemize}
			\item $H_M, H_S$ are subgroups of $A_M, A_S$ respectively and $\gamma: q|_{A_M|H_M}\rightarrow -q|_{A_S|H_S}$ is an isometry;
			\item $K$ is an even lattice with signature $(s_+-m_+, s_- - m_-)$ and discriminant form $-\delta$ where $\delta\equiv (q_{A_M}\oplus -q_{A_S})|_{\Gamma_\gamma^\perp/\Gamma_\gamma}$ and $\gamma_K:q_K\rightarrow (-\delta)$ is an isometry.
		\end{itemize}
	\end{lemma}
	
	We will also use the following simplified version of Lemma \ref{embedding general}:
	
	\begin{lemma}\label{embedding into unimodular}
		Let $M$ be an even lattice of signature $(m_+,m_-)$. The existence of a primitive embedding of $M$ into some unimodular lattice $L$ of signature $(l_+, l_-)$ is equivalent to the existence of a lattice $N$ of signature $(n_+, n_-)$ and discriminant form $q_{A_N}$ such that the following are satisfied:
		\begin{itemize}
			\item $m_++n_+=l_+$ and $m_-+n_-=l_-$.
			\item $A_N\equiv A_M$ and $q_{A_N}=-q_{A_M}.$
		\end{itemize}
	\end{lemma}

	\subsection{2-elementary Lattices}
	
	Let $L$ be any even lattice. An involution $\iota\in O(L)$ determines two eigenspaces \[L_{\pm}:=\{v\in L: \iota(v)=\pm v\}.\] The following hold:
	\begin{itemize}
		\item $\Lneg$ and $\Linv$ are primitive, mutually orthogonal lattices
		\item $H:=L/(\Linv\oplus\Lneg)$ is a 2-elementary group (i.e $H=(\bZ/2)^a$);
		\item H admits embeddings into the discriminant groups $A_{\Linv}$ and $A_{\Lneg}$ such that the diagonal embedding $H\subset A_{\Linv}\oplus A_{\Lneg}$ is isotropic with respect to the finite quadratic form $q_{\Linv}\oplus q_{\Lneg}$ [\cite{nikulin} Sect. 1.5]
	\end{itemize}
	\begin{lemma}\label{2-elem unimod}
		If $\Lambda$ is a unimodular lattice and $\iota\subset O(\Lambda) $ is an involution, then $\Lambda_-$ and $\Lambda_+$ are 2-elementary lattices. 
	\end{lemma}
	\begin{lemma}\label{2-elem}
		Let $L$ be a lattice, and $\iota\subset O(L)$ an involution acting as the identity on $A_L$. Then $\Lneg$ is a 2-elementary lattice.
	\end{lemma}
	
	\begin{proof}
		Let $\Lambda$ be a unimodular lattice such that there exists a primitive embedding $L\hookrightarrow \Lambda$, and let $R=L^\perp.$ Since $\iota$ acts as the identity on $A_L$, we can extend $\iota$ by the identity on $R$. Since $R_-=0$, it holds that $\Lneg\cong \Lambda_{-}$. We deduce by Lemma \ref{2-elem unimod} that $\Lneg$ is 2-elementary.
	\end{proof}
	
	Let $M$ be a 2-elementary lattice. Nikulin proved that $M$ is classified by three natural invariants: the signature $(t_+, t_-)$, the number of generators $a$ of $A_M$ and $\delta\in\{0,1\}$ with $\delta=0$ if and only if the finite quadratic form $q_M: A_M\rightarrow \bQ/2\bZ$ takes values in $\bZ/2\bZ$. 
	\begin{theorem}(Classification of 2-elementary lattices)\label{class2-elem}
		The genus of an even 2-elementary lattice is determined by the invariants $\delta_M, l(M)$ and $sign(M)$. If $M$ is indefinite, then the genus consists of one isomorphism class.
		
		An even 2-elementary lattice $M$ with $\delta_M=\delta$, $l(M)=a$ and $sign(M)=(t_+, t_-)$ exists if and only if the following conditions are satisfied:
		
		\begin{enumerate}
			\item $t_++t_-\geq a$;
			\item $t_++t_-+a\equiv 0 \mod 2$; 
			\item $t_+-t_-\equiv 0 \mod 4 $ if $\delta=0$;
			\item $\delta=0$, $t_+-t_-\equiv 0 \mod 8 $ if $a=0$;
			\item $t_+-t_-\equiv 1 \mod 8 $ if $a=1$;
			\item $\delta=0 $ if a=2, $t_+-t_-\equiv 4 \mod 8$;
			\item $t_+-t_-\equiv 0 \mod 8$ if $\delta=0$ and $a=t_++t_-$.
		\end{enumerate}
	\end{theorem}
	\begin{lemma}\label{M(1/2) is defined}
		Let $M$ be an even lattice of rank $r$ such that $(\bZ/2\bZ)^r\subset A_M$. Then $M(1/2)$ is a well-defined integral lattice. 
	\end{lemma}
	
	\begin{proof}
		For $T:=M(1/2)$ to be well-defined, we must check that the induced symmetric bilinear form  takes integral values. Recall that $(x,y)_T=\frac{1}{2}(x,y)_M$, and thus this is equivalent to showing that $div_M(x)$ is divisible by $2$ for all $x\in M$. Let $G$ be the Gram matrix of $M$; it has integral entries, and has Smith Normal form $$diag(d_1,\dots d_r),$$ with $d_1d_2\dots d_r\neq 0$, and $d_i|d_{i+1}$ for $i=1,\dots r-1$.  Further, $d_1d_2\dots d_i$ is the greatest common divisor all the $i\times i$ minors of $G$. In particular, $d_1$ is the gcd of all entries of $G$. Thus in order to prove that $div_M(x)$ is divisible by 2 for all $x\in M$ it suffices to show that $2|d_1$.
		
		Since $M$ is a free module over a PID, there exists a basis $u_1,\dots u_r$ of $M^*$ such that $v_1=d_1u_1,\dots ,v_r=d_ru_r$ is a basis of $M$.	
		Hence
		\[M^*/M=\frac{\bZ u_1\oplus\dots \bZ u_r}{\bZ v_1\oplus \dots \bZ v_r}\equiv (\bZ/d_1\bZ)\oplus \dots (\bZ/d_r\bZ).\] Since $(\bZ/2\bZ)^r$ is a subgroup of $A_M$, we must have that $d_1$ is divisible by 2; if not, $A_M$ can only contain $(\bZ/2\bZ)^{r-1}$.
		
	\end{proof} 
	\begin{lemma}\label{M(1/2) is odd}
		Let $M$ be an even, 2-elementary lattice with $l(A_M)=r$ where $r$ is the rank and $\delta=1$. Consider the lattice $T:=M(\frac{1}{2})$.  Then $T$ is an odd unimodular lattice. 
	\end{lemma}
	
	\begin{proof}
		
		Assume that $T:=M(\frac{1}{2})$ is even. Then for all $x\in T$, we have that $(x\cdot x)_T=2k$ for some integer $k$. Thus $(x\cdot x)_M=4k$.
		
		Since $\delta=1$, there exists $\gamma\in A_M$ of order 2 with $q(\gamma)_M\in (\bQ\setminus \bZ)/2\bZ$. We can write $\gamma=\alpha+M$ for $\alpha\in M^*$, and so $(\alpha\cdot \alpha)_M\in \bQ\setminus \bZ$. Since $\gamma $ has order 2, we have that $2\gamma\in M$, i.e $2\alpha \in M$ and so $\alpha=\frac{x}{2}$ for some $x\in M$. Finally we see that \[(\alpha\cdot \alpha)_M=\frac{(x\cdot x)_M}{4}=k\] with $k$ an integer, providing a contradiction.
		
	\end{proof}
	
	\begin{lemma}\label{T(2) has delta=1}
		Let $T$ be an odd lattice. Then $M:=T(2)$ has $\bZ/2\bZ\subset A_M$ and $q_M(\xi)\not\in \bZ/2\bZ$, where $\xi$ is a generator of $\bZ/2\bZ\subset A_M$.
	\end{lemma}
	\begin{proof}
		Since $T$ is odd, there exists an element $v\in T$ with $(v,v)_T=2k+1$ for some integer $k$. Thus $(v,v)_M=4k+2\equiv 2 \mod 4. $ 
		
		Write $div_N(v)=d$; then $div_M(v)=2d$. Let $v^*=\frac{v}{2d}$, and consider the class $[v^*]\in A_M$. Then $d[v^*]=[\frac{v}{2}]$ is an element of order 2 in $A_M$, and thus generates a subgroup $$\langle d[v^*]\rangle\cong \bZ/2\bZ\subset A_M.$$
		
		Now $$q(d[v^*])=d^2q(v^*)=d^2\frac{(v,v)_M}{4d^2}=\frac{4k+2}{4}\not\in \bZ/2\bZ.$$
	\end{proof}

\begin{theorem}\cite[Section 1]{MR633161}\label{hyperbolic p-elementary}
	\begin{enumerate}
		\item An even, hyperbolic, $p$-elementary lattice $M$ of rank $r$ with $p\neq 2$, $r\geq3$ is uniquely determined by the number $l(A_M)$.
		\item For $p\neq 2,$ a hyperbolic $p$-elementary lattice with invariants $r, a:=l(A_M)$ exists if and only if the following conditions are satisfied: $a\leq r$, $r\equiv 0 \mod 2$, and $$\begin{cases}
			\text{ if } a\equiv 0 \mod 2, \text{ then } r\equiv 2 \mod 4\\
			\text{ if } a\equiv 1 \mod 2, \text{ then } p\equiv (-1)^{r/2-1}\mod 4.
		\end{cases}$$ Moreover, if $r\not\equiv 2 \mod 8, $ then $r>a>0$.
	\end{enumerate}
\end{theorem}

\section{Cubic Scrolls and their intersection}\label{twisted}
The relative position of two planes contained in a cubic fourfold has been well studied \cite{voisintorelli}. The case of a cubic fourfold containing two cubic scrolls does not seem to appear in the literature. Here we make some comments on the intersection of two such cubic scrolls. This is relevant in the context of \S \ref{symplscrolls}, especially Theorem \ref{scrolls in fourfold}. Arithmetically we can compute the intersection of two cubic scrolls $T_1, T_2\subset X$ as $[T_1]\cdot T_2=\tau$ for $\tau\in \{1,3,5\}$ (Lemma \ref{intersections of scrolls}); here we note that geometrically this corresponds to a syzygetic/azygyetic arrangement of twisted cubic curves.

Let $T_1$, $T_2\subset X$ be two scrolls belonging to two different families in $X$. Then $span(T_1)=H_1,$ $span(T_2)=H_2$ with $H_i\cong \bP^4$ and $H_1\neq H_2$. We are interested in the intersection $T_1\cap T_2$; this is thus contained in $H_1\cap H_2\cap X=S_{12}$, a cubic surface.

Note that $H_1\cap H_2\cap T_i$ has degree 3, since $\eta_X\cdot [T_i]=3$. Since $H_1\cap H_2\cong \bP^3$, it is clear that $H_1\cap H_2\cap T_i=C_i$ is a twisted cubic curve, necessarily belonging to $S_{12}$. We can thus reduce the problem to studying the intersection of two twisted cubic curves on a cubic surface.

\subsection{Twisted cubics on a cubic surface}
There are 72 linear systems of twisted cubic curves on a cubic surface $S$; these are in 1-1 correspondence with sets of mutually disjoint six lines i.e a \textit{sixer} (see for example \cite{twistedcubic}). 

Consider a twisted cubic $C\subset S$; this corresponds to a sixer denoted $E_1,\dots E_6\in S$ such that $E_i^2=-1$, and $E_i\cdot E_j=0$. Blowing down these lines $E_i$ determines $\pi:S\rightarrow \bP^2$, and $C$ is the strict transform of a line $l\subset\bP^2$, i.e $C\in |\pi^*(l)|$. 

By \cite[Lemma 9.1.2]{dolg}, to each sixer there exists a unique root $\alpha\in E_6$ such that $E_i\cdot \alpha=1$ for all $i=1,\dots 6$. Further, in this basis $$\alpha=2C-(E_1+\dots +E_6).$$
Note that if $\{F_1, \dots F_6\}$ is another sixer such that $\{E_1,\dots E_6\}, \{F_1, \dots F_6\}$ form a \textit{double six} then the associated root for $\{F_1,\dots F_6\}$ is $-\alpha$. 

\begin{definition}\cite[Lemma 9.1.5]{dolg}
	Let $C,C'\in S$ be twisted cubic curves such that $|C|\neq |C'|$, $\alpha, \beta$ their associated roots. We say that $C$ and $C'$ are a syzygetic duad if and only if $(\alpha,\beta)\in 2\bZ$.
\end{definition}

\begin{lemma}
	Let $C\in S$ be a twisted cubic, $E_1,\dots E_6$ the associated sixer corresponding to the blow up $\pi:S\rightarrow \bP^2$. Let $C'\in S$ be another twisted cubic such that $|C'|\neq |C|$, with associated sixer $ \{F_1,\dots F_6\}$. Then:
\[	C\cdot C'=\begin{cases}
	2 & \text{ if } C. C' \text{ are azygyetic and } (\alpha,\beta)=-1\\
	3 & \text{ if } C, C' \text{ form a syzygetic duad}\\
	4 & \text{ if } C. C' \text{ are azygyetic and } (\alpha,\beta)=1\\
	5 & \text{ if } \{E_1,\dots E_6\}, \{F_1,\dots F_6\} \text{ form a double sixer.}
\end{cases}\]
\end{lemma}
\begin{proof}
	By \cite[Theorem 9.1.2]{dolg}, one can write down the list of 36 double sixers with corresponding associated roots. If $C'$ corresponds to the sixer $\{F_1,\dots F_6\}$, then $$C'\sim \frac{-K_S+\sum{F_1+\dots F_6}}{3},$$ where $K_S\sim 3C-(E_1+\dots E_6)$. One checks the intersection numbers for the corresponding sixers in Dolgachev's list.
	\end{proof}

\subsection{Theta-characteristics of $g=4$ curve}

The 120 pairs of cubic scrolls as in Theorem \ref{scrolls in fourfold} are in 1-1 correspondence with the 120 tritangent planes to a genus 4 degree six canonical curve $C\subset \bP^3$. We can also compute the intersection numbers by investigating how these planes intersect.

Let $J(C)_2$ denote the points of order 2 of the Jacobian of $C$. Then $J(C)_2\cong (\bF_2)^8$ equipped with the Weyl pairing. Then the space of theta-characteristics of $C$ is an affine space over $J(C)_2$.  

Each odd theta-characteristic $\eta$ defines a tritangent hyperplane $H_\eta$. 

\begin{definition}
	We say that four theta-characteristics $\theta_1,\theta_2, \theta_3, \theta_4:=\theta_1+\theta_2+\theta_3$ are syzygetic if $\sum_{i=1}^4 D_{\theta_i}$ is cut out by a quadric.
\end{definition}

\begin{lemma}\cite[Lemma 1.8]{thetachar} Let $\epsilon\in J(C)_2$ be a two torsion point. Let $H_1, H_2$ be two tritangent planes with contact divisors $D_1, D_2$ such that $H_i\cdot C=2D_i$ and $D_1+D_2\in|\omega_C+\epsilon|$. Then there exists a twisted cubic $T$ passing through the six points in the support of $D_1+D_2$.
	\end{lemma}

Fixing an $\epsilon\in J(C)_2$, the automorphism group of the 120 tritangent planes can be identified with $O(8,\bF_2))^+$ preserving the quadratic form. In particular, if the unique quadric containing $C$ is singular, then $C$ has a distinguished even theta characteristic (it vanishes at the vertex). Note that these curves correspond to del Pezzo surfaces of degree 1; $E_8$ acts on the 240 lines in pairs. Each pair corresponds to a tritangent plane of $C$; thus the automorphism group of the tritangent planes is isomorphic to $W(E_8)/\pm 1$.

	\bibliography{Cubicinvolutions}
	\end{document}